\documentclass[preprint,sort&compress,12pt]{elsarticle}

\usepackage[T1]{fontenc}
\usepackage{algorithm}
\usepackage{algorithmicx}
\usepackage{algpseudocode}
\usepackage{mathrsfs}

\usepackage{amsthm}
\usepackage{amsmath}
\usepackage{amsfonts}
\usepackage{amssymb}

\usepackage{booktabs}
\usepackage{dsfont}
\usepackage{afterpage}

\usepackage{xcolor}

\newtheorem{theorem}{Theorem}
\renewcommand{\vec}[1]{\mathbf{#1}}
\newcommand{\mat}[1]{\mathbf{#1}}
\newcommand{\abs}[1]{|#1|}
\newcommand{\ie}{{\em i.e.}}
\newcommand{\eg}{{\em e.g.}}
\newcommand{\ceil}[1]{\lceil #1 \rceil}

\usepackage[normalem]{ulem}

\journal{Computers \& Mathematics with Applications}

\begin{document}

\begin{frontmatter}

\title{A Stochastic Method for Solving Time-Fractional Differential Equations}

\author[1]{Nicolas L. Guidotti \corref{cor1}} \ead{nicolas.guidotti@tecnico.ulisboa.pt}
\author[1,2]{Juan A. Acebr\'on} \ead{juan.acebron@tecnico.ulisboa.pt}
\author[1]{Jos\'e Monteiro} \ead{jcm@inesc-id.pt}

\affiliation[1]{organization = {INESC-ID, Instituto Superior T\'ecnico, Universidade de Lisboa}, country = {Portugal}}
\affiliation[2]{organization = {Department of Mathematics, Carlos III University of Madrid}, country = {Spain}}

\cortext[cor1]{Corresponding author}

\begin{abstract}

We present a stochastic method for efficiently computing the solution of time-fractional partial differential equations (fPDEs) that model anomalous diffusion problems of the subdiffusive type. After discretizing the fPDE in space, the ensuing system of fractional linear equations is solved resorting to a Monte Carlo evaluation of the corresponding Mittag-Leffler matrix function. This is accomplished through the approximation of the expected value of a suitable multiplicative functional of a stochastic process, which consists of a Markov chain whose sojourn times in every state are Mittag-Leffler distributed. The resulting algorithm is able to calculate the solution at conveniently chosen points in the domain with high efficiency. In addition, we present how to generalize this algorithm in order to compute the complete solution. For several large-scale numerical problems, our method showed remarkable performance in both shared-memory and distributed-memory systems, achieving nearly perfect scalability up to $16,384$ CPU cores.

\end{abstract}

%


\end{frontmatter}


\section{Introduction}\label{sec:intro}

Among the many non-conventional statistical phenomena observed experimentally in the last few years and characterized typically by unusual scaling exponents, the anomalous diffusion stands out for its major impact in a variety of scientific disciplines~\cite{metzler_random_2000}.
To cite just a few relevant examples, it has been already observed in turbulent atmospheres, transport in amorphous solids, diffusion in the cell nucleus, magnetic resonance imaging, search behaviour, and in finance for accurate modelling of the stock price fluctuations.

In the anomalous diffusion, the mean squared displacement (MSD) varies nonlinearly in time, \ie, $\langle x^2(t) \rangle \sim t^{\alpha}$ with $\alpha \ne 1$, during the entire process. Here $x(t)$ is the relative position at time $t$ of a particle with respect to a given reference point. This contrasts with the classical diffusion problems, modelled typically by a Brownian motion, where it is well known that the MSD grows linearly in time. In practice, there are two different types of regimes according to the value of the scaling exponent $\alpha$. The phenomenon is termed as subdiffusion when $\alpha<1$ and, superdiffusion when $\alpha>1$. It has been proposed in the literature several pathways leading to anomalous diffusion. The most commonly found are the presence of long-range correlations, non-identical displacements, or non-finite mean or variance of the probability density function for the trajectory displacements.

Within this context, fractional calculus has been proven to be extremely useful in order to provide realistic models for many real-life processes and phenomena \cite{west_fractional_2016, lischke_what_2020}. The main reason for this is due to fractional derivative operators being in practice non-local operators, and therefore they are especially suited for describing the long-time memory and spatial heterogeneity effects typically found in any anomalous diffusion problem.

The formal solution to the simplest dynamic problem defined by a system of linear fractional rate equations for an $n$-dimensional process $\vec{y}(t)$ subject to a given initial condition $\vec{y}(0)={\bf y_0}$ is

\begin{equation}\label{eq:n-dim-sol}
\vec{y}(t)=E_{\alpha}(\mat{A} \,t^{\alpha}) \, {\bf y_0},
\end{equation}
where $\mat{A}$ is an $n\times n$ matrix of constant coefficients. When $\alpha=1$, the Mittag-Leffler (ML) function $E_{\alpha}(\mat{A} \,t^{\alpha})$ reduces to the matrix exponential $e^{\mat{A} \,t}$. The ML function also appears in other scientific domains, such as biology \cite{estrada_fractional_2020, magin_fractional_2004}, physics \cite{hilfer_applications_2000, sabatier_advances_2007, mainardi_fractional_2010}, control systems \cite{sabatier_advances_2007, chen_fractional_2009} and complex network analysis \cite{arrigo_ml_network_2021, martinez_world-wide_2022}. It is important to remark here that this system of linear rate equations serves as building blocks for more complex systems, and therefore becomes of paramount importance in investigating algorithms capable of computing it efficiently.

However, an efficient algorithm for solving large-scale problems is still missing. Therefore, the main contribution of this paper is to propose a probabilistic method for the efficient computation of the ML matrix function. This method is based on random walks that change states according to a suitable continuous-time Markov chain, whose sojourn times follow a Mittag-Leffler distribution.
Note that a random process governed by the ML distribution is no longer Markovian due to the presence of long-term memory between states.
The method we proposed allows us to compute the ML matrix function for large-scale matrices, focusing primarily on the action of the function over a vector. Our second contribution is an extensive analysis of the performance and convergence of our method using a few relevant numerical examples, including a formal description of the variance and a performance comparison against classical algorithms. Our third and last contribution was to parallelize and analyze the scalability of the stochastic algorithm for a large number of processors (up to $16,384$ CPU cores) in the Karolina Supercomputer located in IT4Innovations National Supercomputing Centre.

The paper is organized as follows. In the next section, we give an overview of the state of the art in this area. Section \ref{sec:method} describes the probabilistic method and its practical implementation. This section also shows how a program can generate random numbers from the Mittag-Leffler probability distribution. In Section \ref{sec:results}, several numerical experiments are presented in order to characterize the performance and numerical errors of the stochastic method compared to classical approaches. Finally, in Section \ref{sec:conclusion}, we highlight the main results and conclude the paper.

\section{Background}\label{sec:background}

One of the major drawbacks of describing a problem in terms of fractional operators is the lack of uniqueness in the mathematical formulation. This contrasts sharply with the classical calculus. In fact, in fractional calculus, there exists a variety of different definitions of fractional operators (differentials and integrals), and they are chosen according to a given set of assumptions imposed to satisfy different needs and constraints of the particular problem. There are already in the literature several excellent texts that review the mathematics of fractional calculus, \eg~\cite{mainardi_fractional_2010}. Among the different definitions of fractional operators, the Caputo fractional derivative appears as the most commonly used for modelling temporal evolution phenomena, which accounts for past interactions and non-local properties.
Moreover, mathematically the initial conditions imposed for fractional differential equations with Caputo derivatives coincide with the integer-order differential equations. Such derivative is defined as follows:
\begin{equation}
^{C}D_t^{\alpha}f(t)=\int_0^t \frac{f^{(n)}(\tau)}{(t-\tau)^{\alpha+1-n}}d\tau,
\label{caputo_def}
\end{equation}
with $n$ an integer such that $ n = \ceil{\alpha}$. Note that if $0<\alpha<1$, $n = 1$.

The simplest dynamic process described in terms of fractional operators consists of an initial value problem for a linear fractional rate equation, which for a given variable $y(t)$ is given by
\begin{equation}
^{C}D_t^{\alpha} y(t)=-\lambda y(t), \quad y(0)=y_0.
\end{equation}
The solution of this equation can not be expressed in a closed form, being rather an infinite series. It was first obtained by the mathematician Mittag-Leffler and is formally written as
\begin{equation}
y(t)=y_0\, E_{\alpha}(-\lambda t^{\alpha}),
\end{equation}
where $E_{\alpha}(t)$ is the function defined as
\begin{equation}
E_{\alpha}(-\lambda\, t^{\alpha})=\sum_{k=0}^{\infty} \frac{(-\lambda\,t^{\alpha})^k}{\Gamma(k\alpha+1)},
\end{equation}
and carries his name. Here $\Gamma(x)$ denotes the Euler's Gamma function. Note that when $\alpha=1$, it reduces to the exponential function. Essentially, the Mittag-Leffler function generalizes the exponential function and has been considered for many as the {\it Queen function} of the fractional calculus~\cite{mainardi_why_2020, gorenflo_mittag-leffler_2020, gorenflo_fractional_1997}. For complex parameters $\alpha$ and $\beta$ with $\mathbb{R}\{\alpha\} > 0$, the Mittag-Leffler function (ML function) can be further generalized, and was defined as
\begin{equation}
\label{eq:ml_def}
    E_{\alpha,\beta}(z) = \sum_{k = 0}^{\infty}{\frac{z^k}{\Gamma(k\alpha + \beta)}}, z \in \mathbb{C}.
\end{equation}
Here $E_{\alpha, \beta}(z)$ is an entire function of order $\rho = 1 / \mathbb{R}\{\alpha\}$ and type 1.  The exponential function can be recovered by setting $\alpha = \beta = 1$ since $\Gamma(k~+~1)~=~k!$ for $k \in \mathbb{N}$. When $\beta = 1$, the notation of the ML function can be simplified to $E_\alpha(z) = E_{\alpha, 1}(z)$.

A naive approach to evaluate the ML function comes directly from its definition in (\ref{eq:ml_def}). However, this is ill-advised for most applications since the convergence of the series is very slow if either the modulus $\abs{z}$ is large or the value of $\alpha$ is small. Moreover, the terms of the series may grow very large before decreasing, which can cause overflows and catastrophic numerical cancellations when using standard double-precision arithmetic (IEEE 754). For this reason, the scalar ML function is usually evaluated through alternative methods~\cite{hilfer_computation_2006, garrappa_numerical_2015, seybold_numerical_2009, gorenflo_computation_2002}, such as the Taylor series, inversion of the Laplace transformation, integral representation and other techniques.

To evaluate the ML matrix function, Garrappa~\cite{garrappa_computing_2018} recently proposed a Schur-Parlett algorithm~\cite{davies_schur-parlett_2003}, which in practice requires computing its derivatives. This method, however, is only suited for small matrices since the Schur-Parlett algorithm scales with $O(n^3)$, but it can go up to $O(n^4)$ depending on the distribution of the eigenvalues. Moreover, if the eigenvalues are highly clustered, the algorithm may require high-order derivatives which are more expensive to compute, while also being less accurate. There is also a Krylov-based method for computing the action of a Mittag-Leffler function over a vector~\cite{moret_convergence_2011}, however, the code is not publicly available.

As an alternative to deterministic methods, Monte Carlo algorithms for approximating matrix functions have already been proposed in the past~\cite{forsythe_matrix_1950, dimov_monte_2008, dimov_parallel_2001, ji_convergence_2013, benzi_analysis_2017}, primarily for solving linear systems. In essence, these methods generate random walks sampled from a discrete Markov chain governed by the matrix $\mat{A}$, approximating each power of the Neumann series \cite{higham_functions_2008}. The convergence of this method was rigorously established in~\cite{benzi_analysis_2017, ji_convergence_2013, dimov_new_2015}. It is broadly accepted that these stochastic methods offer interesting features from a computational point of view, such as being easily parallelizable, fault-tolerant, and more suited to heterogeneous architectures (which are extremely important attributes because of the architecture of current high-performance computers). However, the truth is that they also exhibit some significant weaknesses, namely a very slow convergence rate to the solution. This can cause the underlying algorithms to be highly demanding computationally, especially when highly accurate solutions are required.

Only recently have Monte Carlo methods been extended to compute other matrix functions, namely the matrix exponential \cite{acebron_monte_2019}. The method requires first to decompose the input matrix into two matrices, one of them a diagonal matrix, and the other a Laplacian matrix. The computation of the matrix exponential is then approximated using the Strang splitting. The algorithm consists of generating continuous-time random walks using the Laplacian matrix as the generator, with a holding time governed by an exponential distribution. Later, in ~\cite{acebron_highly_2020} it was applied a multilevel technique to improve the performance of this method. However, the Strang splitting cannot be applied to the ML matrix function, and thus, we have to generate the random walks directly using the input matrix.

For the specific case of the one-dimensional fractional diffusion equation,~\cite{nichols_subdiffusive_2018} introduced a Monte Carlo method for simulating discrete time random walks which are used for obtaining numerical solutions of the fractional equation, but not for computing the ML matrix function.

\section{Description of the Probabilistic Method} \label{sec:method}

In this section, we describe the probabilistic representation of the solution vector as well as a practical algorithm for approximating the action of the Mittag-Leffler function over a vector using this representation. With slight modifications, the algorithm can either calculate a single entry of the solution vector or the full solution. We also present a way to sample random numbers from the Mittag-Leffler distribution.

\subsection{Mathematical description} \label{sec:theory}

Consider: $\mat{A}=(a_{ij})$, a general {\it n}-by-{\it n} matrix with $a_{ii} < 0$ $\forall i$; $\vec{u}$, a given {\it n}-dimensional vector; and $\vec{y}(t)$, an {\it n}-dimensional vector. This last vector corresponds to
the solution vector after computing the action of the Mittag-Leffler function over the vector $\vec{u}$, that is $\vec{y}(t)=E_{\alpha,\beta} (\mat{A}\, t^{\alpha})\,{\bf u}$ with $0<\alpha\le 1$, $\beta>0$ and $t \geq 0$. Here $t$ denotes the value of time when the solution is computed.

Let us define a diagonal matrix $\mat{D}$, represented here as a vector $\vec{d}$ with $d_i = d_{ii} = a_{ii}$ for $i=1,\ldots,n$; a matrix $\mat{M} = (m_{ij})$, a matrix obtained as $\mat{M} = \mat{A} - \mat{D}$, and hence with zero diagonal entries; $\mat{Q}$, the matrix with entries $q_{ij}$ given by
\begin{equation}
q_{ij}=
    \begin{cases}
      0, & \text{if}\ i=j \\
     \frac{|m_{ij}|}{\sum_{j=1}^n|m_{ij}|},  & \text{otherwise;}
    \end{cases}\label{def.Q}
\end{equation}
$\mat{G}$, a matrix with entries $g_{ij}$ taking values $1$ when $m_{ij}\ge 0$, and $-1$ otherwise; and finally $\mat{R}$, a matrix with entries
\begin{equation}
r_{ij}=
    \begin{cases}
      r_i=\frac{\sum_{j=1}^n|m_{ij}|}{d_i}, & \text{if}\ i=j \\
     0,  \quad \text{otherwise;}
    \end{cases}\label{def.R}
\end{equation}
Note that $\sum_{j=1}^n q_{ij}=1 \;\;\forall i$,  and therefore matrix $\mat{Q}$ can be used as the transition rate matrix of a given Markov chain. Furthermore, it holds that $\mat{M} = \mat{D}\,\mat{R}\,(\mat{G}*\mat{Q})$, where $*$ denotes element-wise matrix multiplication.

\begin{theorem}
\label{theorem1}
Let $\{X_t:t\ge 0\}$ be a stochastic process with finite state space $\Omega~=~\{1,2,\dots,n\}$ given by
\begin{equation}
X_t=\sum_{m=1}^{\infty} Z_{m-1}\mathds{1}_{[T_{m-1}\leq t \leq T_{m}]}.
\end{equation}
Such process changes states according to a Markov chain $Z=(Z_m)_{m\in \mathbb{N}}$, which takes values in $\Omega$ and $\mat{Q}$ is the corresponding transition matrix. Here, $T_k$ is the time of the $k$-th event, and $\mathds{1}_E$ denotes the indicator function, being $1$ or $0$ depending on whether the event $E$ occurs. The sojourn times in the $i$-th state follow the Mittag-Leffler distribution $ML_\alpha(d_i)$. Then, we have that any entry $i$ of the solution vector $\vec{y}(t)$ can be represented probabilistically as
\begin{equation}
y_i(t)=\mathds{E} \left[ u_{X_0}\mathds{1}_{[T_0>t]}+ \left(\prod_{j=1}^\eta r_{X_{T_{j-1}}}\, g_{X_{T_{j-1}},X_{T_{j}}} \right) u_{X_{T_\eta}} \mathds{1}_{[T_0\leq t]}\right],\label{prob_rep}
\end{equation}
where $X_0=i$. Here $\mathds{E}$ is the expected value with respect to the joint distribution function of the random variables $T$ and $\eta$, where $\eta$ is the number of events occurring between $0$ and $t$.
 \end{theorem}

\begin{proof}
Our goal is to prove that the expected value of the functional of the stochastic process $X_t$ in (\ref{prob_rep}) coincides with the vector solution $\vec{y}(t)$. As a first step in the proof, it is required to compute the joint distribution function of the random variables $T$ and $\eta$. This function depends on the probability of jumping from state $i$ to state $j$ at time $t$, $\mathds{P}(X_t=j|X_0=i)$. Let us define $\mat{P}$ as the probability matrix with entries $p_{ij}$ given by $p_{ij}=\mathds{P}(X_t=j|X_0=i)$, then it holds that
\begin{equation}
p_{ij}(t)=\delta_{ij}\,\mathds{P}(t_i>t)+\int_0^t ds\, f_i(s) \sum_{l\ne i} \mathds{P}(X_\tau=l,X_t=j|\tau=s),
\label{prob_def1}
\end{equation}
where $t_i$ corresponds to the instant of time of a first event when the state is $i$.
Here the function $f_i(s) = -d_i\, s^{\alpha-1}\,E_{\alpha,\alpha}(d_i\, s^{\alpha})$  is a density function. In fact, it satisfies that
\begin{equation}
\int_0^{\infty}ds\,f_i(s)=1,
\end{equation}
and is called the Mittag-Leffler density function \cite{gorenflo_mittag-leffler_2020}. Then the probability of no events taking place in the interval of time $[0,t]$,  $\mathds{P}(t_i>t)$,  is given by
\begin{equation}
\mathds{P}(t_i>t)=\int_t^{\infty}ds\,f_i(s)=E_{\alpha}(d_i\, t^{\alpha}).
\end{equation}
Therefore, it follows that
\begin{equation}
p_{ij}(t)=\delta_{ij}\,E_{\alpha}(d_i\, t^{\alpha})+\int_0^t ds\, f_i(s) \sum_{l\ne i} q_{il}\, p_{lj}(t-s).
\label{prob_def2}
\end{equation}
This is an integral equation for $p_{ij}(t)$ that can be solved recursively, and the solution can be rewritten formally as
\begin{equation}
p_{ij}(t)=\sum_{\eta=0}^\infty p_{ij}^{(\eta)}(t),
\end{equation}
where $p_{ij}^{(\eta)}(t)$ denotes the probability of having a transition from $i$ to $j$ when $\eta$ events occur during the time interval $[0, t]$. Let $\nu_i^{(\eta)}(t)$ be the corresponding contribution to the solution $y_i(t)$ for $\eta$ events. If no events occur, the contribution to the solution $\nu_i^{(0)}(t)$ can be readily obtained as $E_{\alpha}(d_i\, t^{\alpha})\,u_i$. The remaining contributions can be obtained as follows. For one event, from (\ref{prob_def2}) the probability is given by
\begin{equation}
p_{ij}^{(1)}(t)=\int_0^t ds\, f_i(s) \sum_{l\ne i}  q_{il}\,\delta_{lj}\, E_{\alpha}(d_l\, (t-s)^{\alpha}),
\end{equation}
and thus the corresponding contribution to the solution is
\begin{equation}
\nu_i^{(1)}(t)=\sum_{j=1}^n \int_0^t ds\, f_i(s)\sum_{l\ne i} r_i\, g_{il}\, q_{il}\, \delta_{lj} E_{\alpha}(d_l\, (t-s)^{\alpha})\,u_j.
\end{equation}
Note that the equation above can be rewritten equivalently in matrix form as
\begin{equation}
\boldsymbol{\nu}^{(1)}(t)=\int_0^t ds\, {\bf f}(s) \mat{R}\,(\mat{G}*\mat{Q})\, E_{\alpha}(\mat{D}\, (t-s)^{\alpha})\,{\bf u},
\end{equation}
where $\boldsymbol{\nu}^{(1)}$ and ${\bf f}(s)$ are vectors with components $\nu_i^{(1)}$ and $f_i(s)$, respectively. Since  $\mat{R}\,(\mat{G}*\mat{Q}) = \mat{D}^{-1}\,\mat{M}$, then
\begin{equation}
\boldsymbol{\nu}^{(1)}(t)=-\int_0^t ds\,  s^{\alpha-1}\,E_{\alpha,\alpha}(\mat{D}\, s^{\alpha}) \mat{M}\, E_{\alpha}(\mat{D}\, (t-s)^{\alpha})\,{\bf u}.
\end{equation}
From (\ref{prob_def2}), the contribution corresponding to the two events is
\begin{equation}
\boldsymbol{\nu}^{(2)}(t)=\int_0^t ds_1 {\bf f}(s_1)\, \mat{D}^{-1} \mat{M}\, \int_0^{t-s_1} ds_2  {\bf f}(s_2)\,\mat{D}^{-1} \mat{M}\, E_{\alpha}(\mat{D}\, (t-s_1-s_2)^{\alpha})\,{\bf u},
\end{equation}
and can be rewritten in terms of $\boldsymbol{\nu}^{(1)}$ as follows
\begin{equation}
\boldsymbol{\nu}^{(2)}(t)=-\int_0^t ds {\bf f}(s)\, \mat{D}^{-1} \mat{M}\, \boldsymbol{\nu}^{(1)}(t-s).
\end{equation}
Similarly, for an arbitrary number of events $\eta$, the corresponding contribution can be computed in terms of the contribution 
for $\eta-1$ events, and yields
\begin{equation}\label{general_term}
\boldsymbol{\nu}^{(\eta)}(t)=-\int_0^t ds {\bf f}(s)\, \mat{D}^{-1} \mat{M}\, \boldsymbol{\nu}^{(\eta-1)}(t-s).
\end{equation}
Let define an integral operator $\mat{H}$ as
\begin{equation}
\mat{H}\psi(t):= \int_0^t ds\, {\bf f}(s) \mat{D}^{-1} \mat{M} \psi(t-s).
\end{equation}
Then, we have that
\begin{equation}
\boldsymbol{\nu}^{(\eta)}(t)=-\mat{H} \boldsymbol{\nu}^{(\eta-1)}(t).
\end{equation}
This recursive equation can be easily solved, and hence the contribution of $\eta$ events is given by 

\begin{equation}
\boldsymbol{\nu}^{(\eta)}(t)=(-1)^\eta\, \mat{H}^\eta E_{\alpha}(\mat{D}\,t)\,{\bf u},\quad \forall \eta\geq 0.
\end{equation}
The solution $\vec{y}(t)$ is finally obtained by summing all those partial contributions, and yields
\begin{equation}
\vec{y}(t)=\sum_{\eta=0}^\infty \boldsymbol{\nu}^{(\eta)}(t)=(\mat{I} + \mat{H})^{-1}\, E_{\alpha}(\mat{D}\,t)\,{\bf u},
\end{equation}
which can be expressed equivalently as an integral equation in the form
\begin{equation}
\vec{y}(t)=E_{\alpha}(\mat{D}\,t)\,{\bf u}-\int_0^t ds\, {\bf f}(s) \mat{D}^{-1} \mat{M}\, {\vec{y}}(t-s).
\end{equation}
Taking the Laplace transform, and using that~\cite{gorenflo_mittag-leffler_2020}
\begin{equation}
 \mathscr{L}\{t^{\beta-1} E_{\alpha,\beta}(c\,t^\alpha)\}=s^{-\beta} (1-c\,s^{-\alpha})^{-1},\label{Laplace_MLF}
\end{equation}
we obtain
\begin{equation}
{\hat {\vec{y}}}(s)=s^{-1}(\mat{I} - \mat{D}\,s^{-\alpha})^{-1}\,{\bf u}+\mat{D}\,s^{-\alpha}(\mat{I}- \mat{D}\,s^{-\alpha})^{-1}\, \mat{D}^{-1} \mat{M}\,{\hat{\vec{y}}(s)},
\end{equation}
where ${\hat {\vec{y}}}(s)=\mathscr{L}\{y(t)\}$. The solution is
\begin{equation}
{\hat {\vec{y}}}(s)=s^{-1}(\mat{I}- \mat{D}\,s^{-\alpha}- \mat{D}\,s^{-\alpha}\, \mat{D}^{-1} \mat{M})^{-1} \,{\bf u}.
\end{equation}
After applying the inverse Laplace transform and using it again (\ref{Laplace_MLF}), we obtain
\begin{equation}
{\vec{y}}(t)=E_{\alpha}((\mat{D}+\mat{M})t^{\alpha})\,{\bf u}=E_{\alpha}(\mat{A}\,t^{\alpha})\,{\bf u},
\end{equation}
which concludes the proof.
\end{proof}
\subsection{Implementation} \label{sec:implementation}

\begin{algorithm}[t]
\caption{Calculates the $i$-th entry of $\vec{y}(t) = E_\alpha (\mat{A} \, t^\alpha) \, \vec{u}$ based on Theorem~\ref{theorem1}. $N_p$ indicates the number of random paths.} \label{code:ml_ctrw}

\begin{algorithmic}[1]
\Function{MCML\_Single}{$\mat{A}$, $\vec{u}$, $\alpha$, $i$, $N_p$, $t$}

\State $y_i = 0$

\For{\textbf{each} random path}
    \State $X_0 = i; \omega = 1$; $k = 0$
    \State Generate $Z_\alpha \sim ML_\alpha(d_{X_0})$
    \State $\tau = Z_\alpha$

    \While{$\tau < t$}
        \State $X_{k + 1} =$ \Call{SelectNextState}{$X_k$, $\mat{Q}$}
        \State $\omega = \omega \times r_{X_k} g_{X_k, X_{k + 1}}$
        \State Generate $Z_\alpha \sim ML_\alpha(d_{X_k})$
        \State $\tau = \tau + Z_\alpha$
        \State $k = k + 1$
    \EndWhile

    \State $y_i = y_i + \omega \, u_{X_k}$
\EndFor
\State $y_i = y_i / N_p$
\State \Return $y_i$
\EndFunction

\end{algorithmic}
\end{algorithm}

According to Theorem~\ref{theorem1}, the value of the $i$-th entry of the solution vector $\vec{y}(t)$ can be estimated through the simulation of the stochastic process $X_t$, which consists of generating random paths from the Markov chain $Z$ and then computing the realization of a random variable $\omega$ over these paths. Each random path starts at state $X_0$ and time $\tau = 0$ and then jumps from one state to the next until it reaches the time $\tau = t$. The next state is always selected at random based on the transition matrix $\mat{Q}$ and the sojourn time in each state follows an ML distribution. For a sequence of states $X_0, X_1, \ldots, X_\eta$, the value of $\omega$ associated with this random path can be calculated as

\begin{equation*}
    \omega = \left (\prod_{j=1}^\eta r_{X_{j-1}}\, g_{X_{j-1},X_{j}} \right) u_{X_\eta}.
\end{equation*}

The procedure for simulating the process $X_t$ is described in Algorithm~\ref{code:ml_ctrw}. In order for the computation to be practical, Algorithm~\ref{code:ml_ctrw} can only generate a finite number $N_p$ of random paths and approximate the value of $y_i(t)$ as
\begin{equation}
    y_i(t) \approx \frac{1}{N_p} \sum_{j = 1}^{N_p}{\omega_j},
    \label{eq:single_solution}
\end{equation}
where $\omega_j$ is the realization of $\omega$ over the random path $j$. Naturally, replacing the expected value in (\ref{prob_rep}) by the arithmetic mean over a finite sample size $N_p$ in (\ref{eq:single_solution}) leads to a numerical error $\epsilon$ of order of $O(N_p^{-1/2})$. In fact, it is well-known that the arithmetic mean provides the best unbiased estimator for the expected value, and, for large number $N_p$ of random paths, the error $\epsilon$ is approximately a random variable distributed according to a normal distribution with standard deviation $\epsilon \approx \sigma N_p^{-1/2}$ \cite{berry_accuracy_1941}. Here $\sigma$ denotes the standard deviation of the random variable $\omega$. Note that there is no other source of error in Algorithm~\ref{code:ml_ctrw}.

The main limitation of Algorithm \ref{code:ml_ctrw} is that it can only estimate one entry of $\vec{y}(t)$ at a time, requiring a separated set of random paths to calculate the solution in each point in the domain. Therefore, it is better suited for estimating the solution locally (\ie, for a specific set of a number of points) than the complete solution. A more efficient method for computing the solution for the entire domain is presented next.

\subsection{Computing the full solution vector} \label{sec:full_vec_method}

\begin{algorithm}[t]
\caption{Calculates $\vec{y}(t) = E_\alpha (\mat{A} \, t^\alpha) \, \vec{u}$ based on its full probabilistic representation. $N_p$ indicate the number of random paths.} \label{code:ml_ctrw_full}

\begin{algorithmic}[1]
\Function{MCML\_Full}{$\mat{A}$, $\vec{u}$, $\alpha$, $N_p$, $t$}

\State $\vec{y} = 0$
\State $\vec{p}^{(0)} = \left \{ p_i^{(0)} = \cfrac{|u_i|}{\|\vec{u}\|_1} \right \}$

\For{\textbf{each} random path}
    \State $k = 0$
    \State $X_0 =$ \Call{SelectInitialState}{$\vec{p}^{(0)}$}
    \State $\omega = \cfrac{u_{X_0}}{|u_{X_0}|} \|\vec{u}\|_1$
    \State Generate $Z_\alpha \sim ML_\alpha(d_{X_0})$
    \State $\tau = Z_\alpha$

    \While{$\tau < t$}
        \State $X_{k + 1} =$ \Call{SelectNextState}{$X_k$, $\mat{Q}^\intercal$}
        \State $\omega = \omega \times r_{X_k} g_{X_k, X_{k + 1}}$
        \State Generate $Z_\alpha \sim ML_\alpha(d_{X_k})$
        \State $\tau = \tau + Z_\alpha$
        \State $k = k + 1$
    \EndWhile

    \State $y_{X_k} = y_{X_k} + \omega$
\EndFor
\State $\vec{y} = \vec{y} / N_p$
\State \Return $\vec{y}$
\EndFunction

\end{algorithmic}
\end{algorithm}

Theorem~\ref{theorem1} can be conveniently modified to represent the complete solution vector $\vec{y}(t)$ rather than a single entry. The procedure is identical to the case of the matrix exponential (namely $\alpha=\beta=1$), which is described fully in~\cite{acebron_highly_2020}. For completeness, we present here the main details. The representation requires generating random paths, which now start at a randomly chosen state $j$ and time $t=0$ according to a suitable distribution probability $p_j^{(0)}=\mathds{P}(X_0=j)$, ending at state $i$ at time $t$. Those paths evolve in time-changing states governed now by $\mat{Q}^\intercal$ as the corresponding transition matrix of the Markov chain. This can be explained by resorting to Bayes' theorem. In fact, it holds that
\begin{align*}
    p_{ji} &= \mathds{P}(X_t=i|X_0=j) \\
           &= \mathds{P}(X_t=j|X_0=i) \frac{\mathds{P}(X_t=i)}{\mathds{P}(X_0=j)} \\
           &= p_{ij} \frac{p_i}{p_j^{(0)}} \\
           &= (\mat{P}^\intercal)_{ji} \frac{p_i}{p_j^{(0)}},
\end{align*}
where $p_i=\mathds{P}(X_t=i)$. Note that the matrix probability $(\mat{P}^\intercal)$ corresponds to a Markov chain governed now by the transition matrix $\mat{Q}^\intercal$. As it was explained in~\cite{acebron_highly_2020}, the distribution function $p_j^{(0)}$ can be arbitrarily chosen. However, the choice may have a direct impact on the variance, and, in turn, on the performance of the algorithm. For simplicity, we use here the more reasonable choice corresponding to $p_j^{(0)}$ proportional to $|u_j|$, since this resembles the well-known importance sampling method for variance reduction, where the sampling is done according to the importance of the data.

Algorithm \ref{code:ml_ctrw_full} describes the stochastic method based on this alternative representation. The main advantage of this method is the ability to estimate the entire solution vector with a single set of random paths, naturally distributing the contribution of each random path over all points in the domain. Due to this property, Algorithm \ref{code:ml_ctrw_full} is often more efficient than Algorithm \ref{code:ml_ctrw} for estimating the entire solution.

\subsection{Generating random numbers from a Mittag-Leffler distribution}

For every state in the random path, Algorithms~\ref{code:ml_ctrw} and~\ref{code:ml_ctrw_full} generate a random number $Z_\alpha$ from the ML probability distribution to determine the time spent in this state. The value of $Z_\alpha$ is then later added to the total time $\tau$ of the random path. The random path ends when $\tau \geq t$.

Considering that an ML random variable follows the Geometric Stable Laws~\cite{kozubowski_univariate_1999, kozubowski_fractional_2001}, it can be represented as a combination of an exponential and a stable random variable~\cite{kozubowski_exponential_2000, kozubowski_computer_2000}. Using this representation, Kozubowski~\cite{kozubowski_computer_2000} deducted the following formula for generating $Z_\alpha \sim ML_\alpha(\gamma)$ for $0~<~\alpha~\leq~1$:
\begin{equation}
    \label{eq:ml_rand}
    Z_\alpha = - |\gamma|^{-\frac{1}{\alpha}} \, \ln{(U)} \, [\sin(\alpha\pi) \cot(\alpha \pi V) - \cos(\alpha\pi)]^{\frac{1}{\alpha}},
\end{equation}
where $U, V$ are random numbers from the uniform distribution over the $[0, 1]$ interval and $\gamma$ is the rate parameter. For $\alpha = 1$, (\ref{eq:ml_rand}) reduces to the inversion formula for the exponential distribution: $Z_1 = - |\gamma|^{-1} \ln{U}$. Note that for $\alpha > 1$, the ML function is not monotonic and, thus, cannot be considered a probability distribution~\cite{pillai_mittag-leffler_1990}. It is worth mentioning that the \textit{rejection sampling} technique~\cite{devroye_non-uniform_1986, marsaglia_ziggurat_2000, rubin_efficient_2006} can also be used for generating the ML random numbers, however, it requires the construction of a pointwise representation directly from (\ref{eq:ml_def}), which can be quite expensive due to the slow convergence of the series~\cite{fulger_random_2013}.

\section{Numerical Results}\label{sec:results}

In this section, we evaluate the performance of our Monte Carlo method by solving time-fractional partial differential equations (fPDEs) through the method of lines \cite{schiesser_numerical_1991, guo_fractional_2015}. In this method, the spatial variables are discretized, transforming the original problem into a system of coupled fractional ordinary differential equations. The system can then be solved by computing the action of the Mittag-Leffler function over the discretized initial values. With the method of lines, we solve the Dirichlet boundary-value problem for both a 2D time-fractional diffusion equation and a 3D time-fractional convection-diffusion equation. We conclude with an example that uses a more complex geometry and the Neumann boundary conditions. 

\subsection{Setup}

All simulations regarding the shared-memory architecture were executed on a commodity workstation with an AMD Ryzen 5800X 8C @4.7GHz and 32GB of RAM, running Arch Linux. The \texttt{MCML\_Full} (Algorithm \ref{code:ml_ctrw_full}) was implemented in C++ with OpenMP and uses \texttt{PCG64DXSM} \cite{oneill_pcg_2014} as its random number generator. \texttt{PCG64DXSM} is an improved version of the default generator of the popular NumPy \cite{numpy} module for Python. The code was compiled with the Clang/LLVM v14.0.0 with the \texttt{-O3} and \texttt{-march=znver3} flags. We will refer to the Monte Carlo method that uses dense and sparse matrices as {\tt mc\_dense} and {\tt mc\_sparse}, respectively.

To the best of our knowledge, the only freely available code capable of computing the Mittag-Leffler function for matrices was proposed by Garrappa and Popolizio~\cite{garrappa_computing_2018} and it is written in MATLAB 2021a. This code is based on the Schur-Parlett algorithm~\cite{davies_schur-parlett_2003, higham_functions_2008} that was implemented in MATLAB as the {\tt funm} command. The Schur-Parlett algorithm consists in decomposing the matrix $\mat{A}$ as $\mat{V}\mat{T}\mat{V}^{*}$ (Schur decomposition~\cite{golub_matrix_2013}) and then computing the Mittag-Leffler function on the upper triangular $\mat{T}$ matrix using the Taylor series and the Parlett recurrence. The derivatives for the Taylor series are calculated through a combination of series expansion, numerical inversion of the Laplace transform and summation formula. We will refer to Garrappa and Popolizio's code as \texttt{matlab}.

Note that the MATLAB backend is written in C/C++ and calls the Intel Math Kernel Library~\cite{intel_mkl} for many matrix and vector operations, which exploits very efficiently the hardware resources of modern CPUs, including multi-threading and SIMD units. It is worth remembering that we only compare our code against a MATLAB implementation due to the lack of any parallel code freely available in C/C++.

Nevertheless, there are some noticeable differences between the algorithms. First of all, \texttt{matlab} only supports dense matrices and has a fixed precision (IEEE 754 double-precision standard), while ours supports both dense and sparse matrices as well as an user-defined precision. Moreover, a Monte Carlo method is not only fully parallelizable but also allows the computation of single entries of the vector solution, which is not possible with the classical deterministic algorithm. Due to the random nature of the Monte Carlo algorithms, all results reported in this section correspond to the mean values obtained after several runs. With the exception of the variance analysis, the program calculates the solution for all points in the domain in all simulations.

\subsection{Time-fractional diffusion equation} \label{sec:poisson}

The first example we consider consists of solving the 2D time-fractional diffusion equation,
\begin{equation}
\label{eq:heat}
D^\alpha_t u(\vec{x}, t) = \nabla^2 u(\vec{x}, t),
\end{equation}
in a domain $\Omega = [-\mu, \mu]^2$, with time $t > 0$,  a space variable $\vec{x}~=~(x, y)~\in~\mathbb{R}^2$, an initial condition $u(\vec{x}, 0)=u_0(\vec{x})$ and a Dirichlet boundary condition $u(\vec{x}, t)\lvert_{\partial\Omega}\ = 0$. Here $D^\alpha_t$ is the Caputo's fractional derivative of order $\alpha$ with $0<\alpha \leq 1$. Considering a discrete mesh with $m$ cells in each dimension, such that $\Delta x = \Delta y = 2\mu/m$, and the standard 5-point stencil finite difference approximation, the approximated solution $\hat{u}(\vec{x}, t)$ for (\ref{eq:heat}) can be written as
\begin{equation}
\label{eq:heat_solution}
\hat{u}(\vec{x}, t) = E_\alpha \left (\frac{m ^ 2}{4\mu^2} \, \hat{\mat{L}} \, t^\alpha \right) \, \hat{u}_0(\vec{x}) = E_\alpha(\mat{A} \, t^\alpha) \, \hat{u}_0(\vec{x}),
\end{equation}
where $\hat{\mat{L}}$ denotes the discrete Laplacian operator. Here, we consider that the off-diagonal entries of $\hat{\mat{L}}$ are positive and the diagonal negative.

%


This example is particularly suited for analyzing the numerical errors since the eigenvalues of the matrix $\mat{A}$ can be calculated analytically~\cite{strang_computational_2007} and, consequently, the corresponding solution of the equation. In fact, for a point $\vec{x} = (x, y)$ in the discrete domain, the corresponding eigenvalue $\lambda(\vec{x})$ and eigenvector $\vec{v}(\vec{x})$ are known to be equal to
\begin{align*}
    \lambda(\vec{x}) &= \frac{m^2}{4\mu^2} \left [ 2 \cos \left (\frac{x \pi}{m + 1} \right ) + 2 \cos \left (\frac{y \pi}{m + 1} \right ) - 4 \right ] \\
   \vec{v}(\vec{x}) &= \left \{ v_{(i, j)} = \sin \left ( \frac{ix \pi}{m + 1}  \right ) \sin \left ( \frac{jy \pi}{m + 1} \right ) \text{ for } i, j = 1, 2, \ldots, m \right \}.
\end{align*}

After organizing the eigenvalues and eigenvectors of $\mat{A}$ into the matrices $\mat{D}$ and $\mat{V}$, respectively, the solution for the 2D diffusion problem can be calculated as

\begin{equation}
    \hat{u}(\vec{x}, t) = \mat{V} \, E_\alpha (\mat{D} \, t^\alpha) \, \mat{V}^\intercal \, \hat{u}_0(\vec{x}).
    \label{eq:poisson_analytic_solution}
\end{equation}

Moreover, we can determine the stiffness ratio $r = \lambda_{max} / \lambda_{min}$~\cite{lambert_computational_1973} for this problem, which is given asymptotically by $r \sim 4m^2 / \pi^2$ for large values of $m$. In particular, the diffusion problem can be considered stiff for the range of values of $m$ used in this section and, thus, is a suitable example for testing the performance of our method in solving this class of problems.

For all the results shown in this section, we chose the domain to be $\Omega = [-1, 1]^2$ and the initial condition $\hat{u}_0(\vec{x}) = c \, m^2 \, \delta (\vec{x}-\vec{x}_c)$, which consists on a discrete impulse located at the centre of the computational grid, $\vec{x}_c~=~(m / 2, m / 2)$, and strength $c=1/4096$. Since the {\tt matlab} code only works with dense matrices, it requires around $25$GB of memory when calculating the solution for a discrete mesh with $m = 160$. In comparison, {\tt mc\_dense} uses around $10$GB of memory for the same mesh, while {\tt mc\_sparse}, only a couple of megabytes. For this reason, the maximum number of cells $m$ in any experiment will be equal to $160$. Regarding the time variable $t$, we choose $t = 0.1$ where it is possible to clearly distinguish the diffusion process for different values of $\alpha$.

\begin{figure}[t]
\centering
\begin{minipage}[t]{0.475 \textwidth}
	\centering	\includegraphics[width=\textwidth]{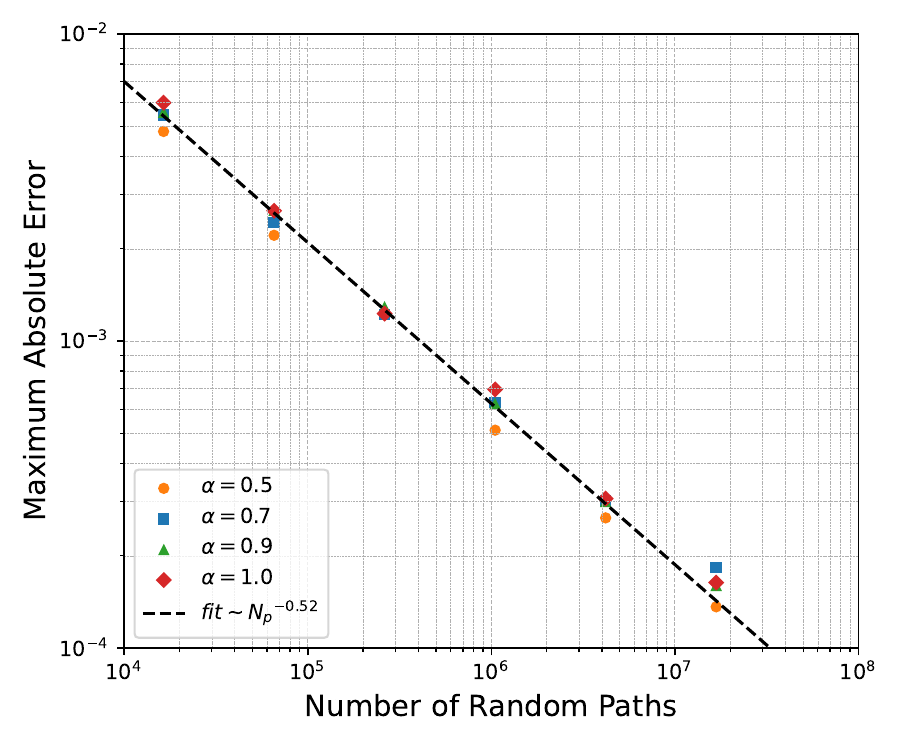}
	\caption{Maximum absolute error as a function of the number of random paths when solving the 2D diffusion equation for $t = 0.1$ and $m = 80$.}
	\label{fig:poisson_err}
\end{minipage} %
\quad
\begin{minipage}[t]{0.475 \textwidth}
	\centering	\includegraphics[width=\textwidth]{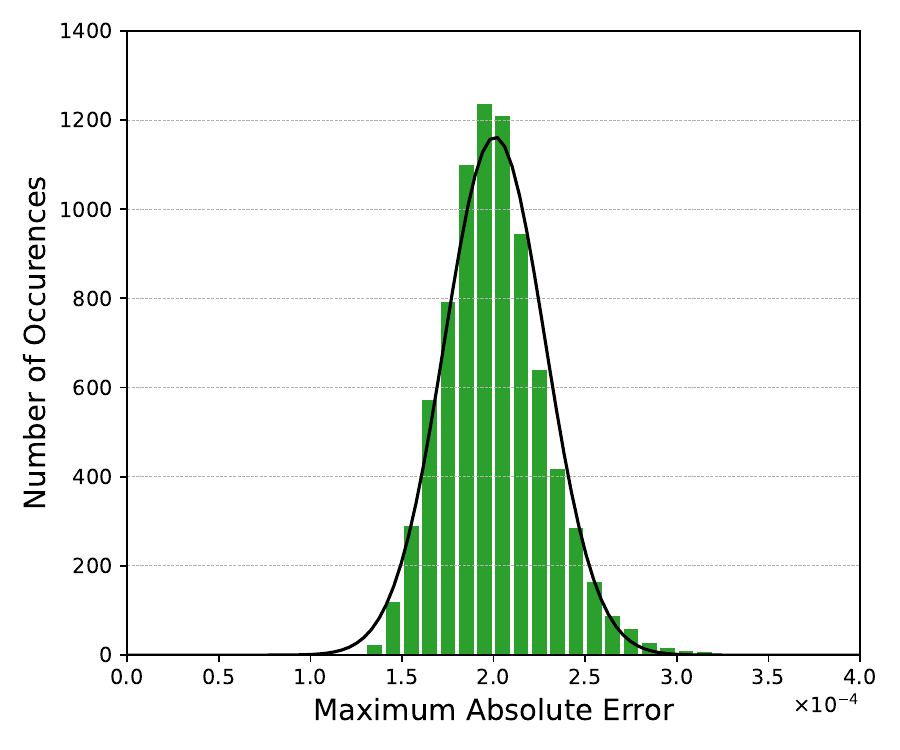}
	\caption{A histogram of the maximum absolute error for $8000$ runs, each one starting with a different random seed. In each run, we solved the 2D diffusion equation with $t = 0.1$, $m = 80$ and $N_p = 2.5 \times 10^{5}$.}
	\label{fig:poisson_err_dist}
\end{minipage}
\end{figure}

Fig.~\ref{fig:poisson_err} shows the maximum absolute error as a function of the number of random paths $N_p$. The error has been computed using the analytical solution in (\ref{eq:poisson_analytic_solution}). Recall from Section~\ref{sec:implementation} that the numerical error of the Monte Carlo method approximates a Gaussian random variable with standard deviation determined by $\sigma \, N_p^{-1/2}$. This relation is confirmed numerically by the trend line in the graph, which has a slope of approximately $-1/2$ in the logarithmic scale. Repeating the Monte Carlo simulation several times with different initial seeds, we observe in Fig.~\ref{fig:poisson_err_dist} that the numerical errors are indeed distributed according to a normal distribution.

\begin{figure}[t]
	\centering	\includegraphics[width=\textwidth]{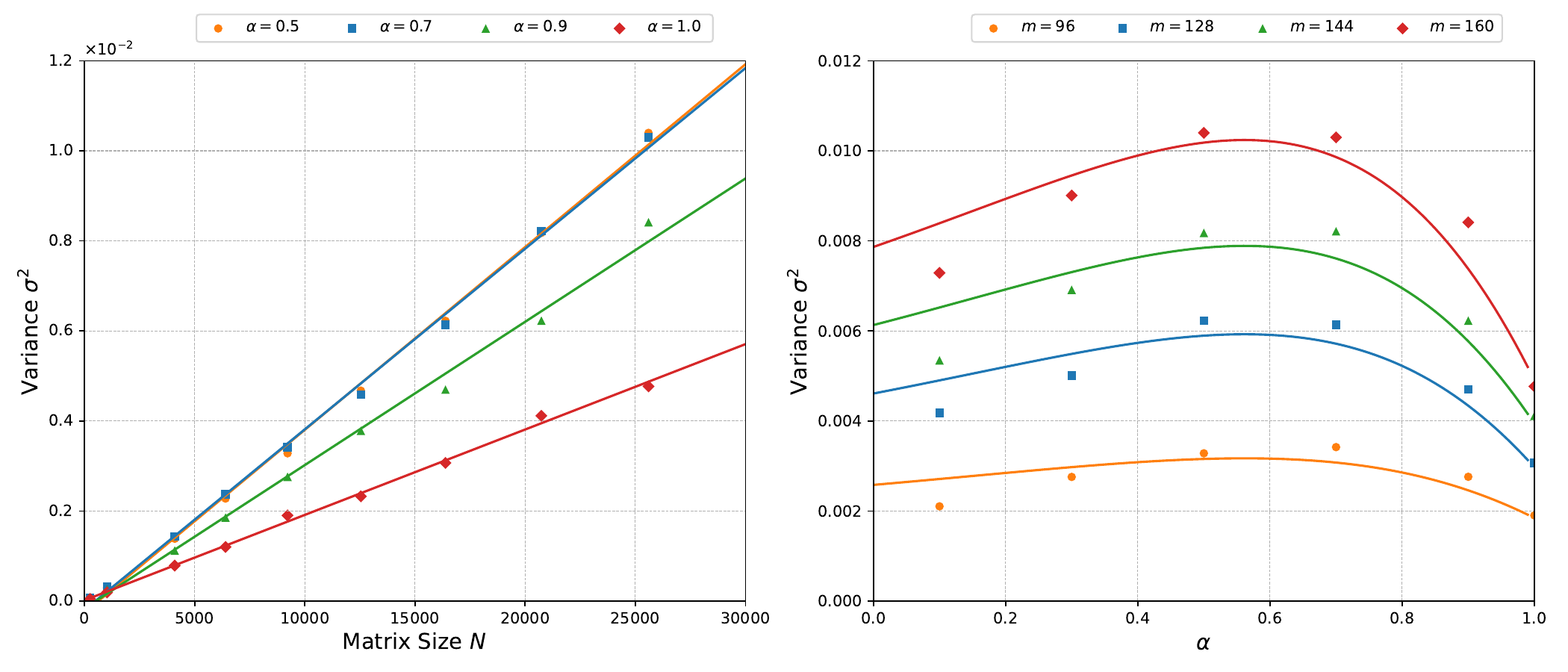}
	\caption{Variance $\sigma^2$ between the random paths as a function of the matrix size $N = m^2$ and $\alpha$ when solving the 2D diffusion equation at a single point located at the centre of the mesh for $t = 0.1$. The number of random paths was kept fixed to $10^{10}$.}
	\label{fig:poisson_scale_var}
\end{figure}

An estimation of the variance can be obtained for this simple initial condition, and in the limiting case when $m\gg 1$. Note that for this case there are only two possible outcomes for the random variable $\omega$ in (\ref{eq:single_solution}), $0$ or $c\, m^2$, with probabilities $1-\xi$ and $\xi$, respectively. The outcome $c\, m^2$ is obtained when the ending state of the random path coincides with the node of the computational mesh corresponding to the point $\vec{x}_c$, being $0$ otherwise. For simplicity, let us assume that we choose the point where the solution is computed to be the same $\vec{x}_c$. Then, the outcome will be different from $0$ only when,
after a random number of jumps $n$ in a random path, the final state is the same as the initial one. Estimating the probability of $n$ is straightforward, since this problem is equivalent to the problem of computing the probability of returning to the initial point after $n$ steps for a simple symmetric 2D random walk \cite{ibe_elements_2013}, and is given by
\begin{equation}
h_n=\left[\binom{n}{n/2}\frac{1}{2n}\right]^2, \quad n=0,2,4,\ldots.
\end{equation}

\begin{figure}[t]
\centering
\begin{minipage}[t]{0.475 \textwidth}
	\centering	\includegraphics[width=\textwidth]{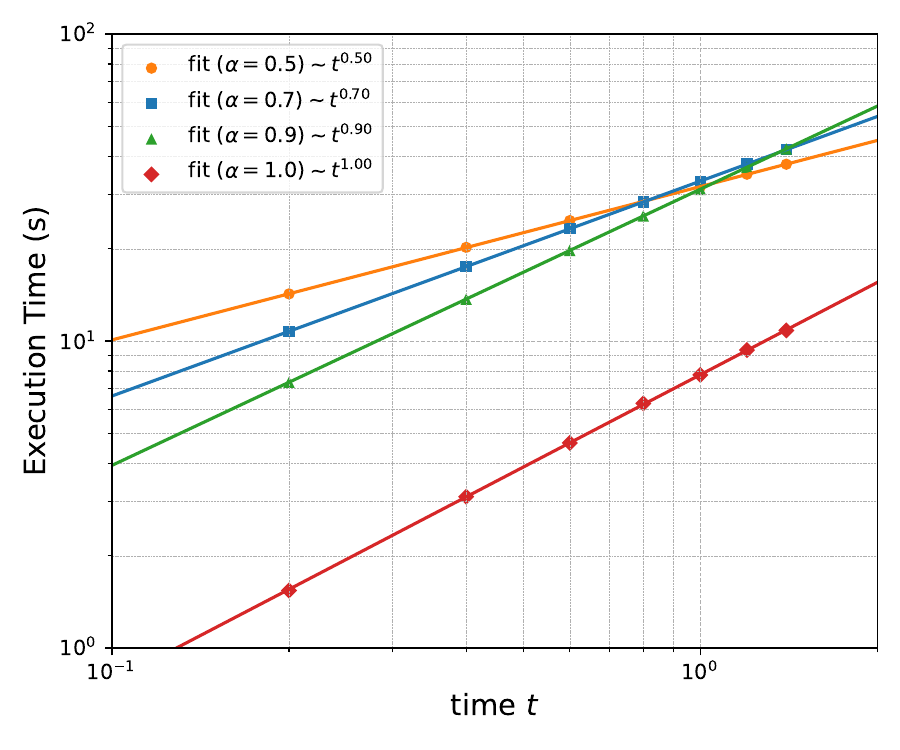}
	\caption{Elapsed time for solving the 2D diffusion equation as a function of the time $t$ and the parameter $\alpha$ for $m = 80$. The number of random paths was kept fixed at $10^5$.}
	\label{fig:poisson_time_scale}
\end{minipage} %
\quad
\begin{minipage}[t]{0.475 \textwidth}
	\centering	\includegraphics[width=\textwidth]{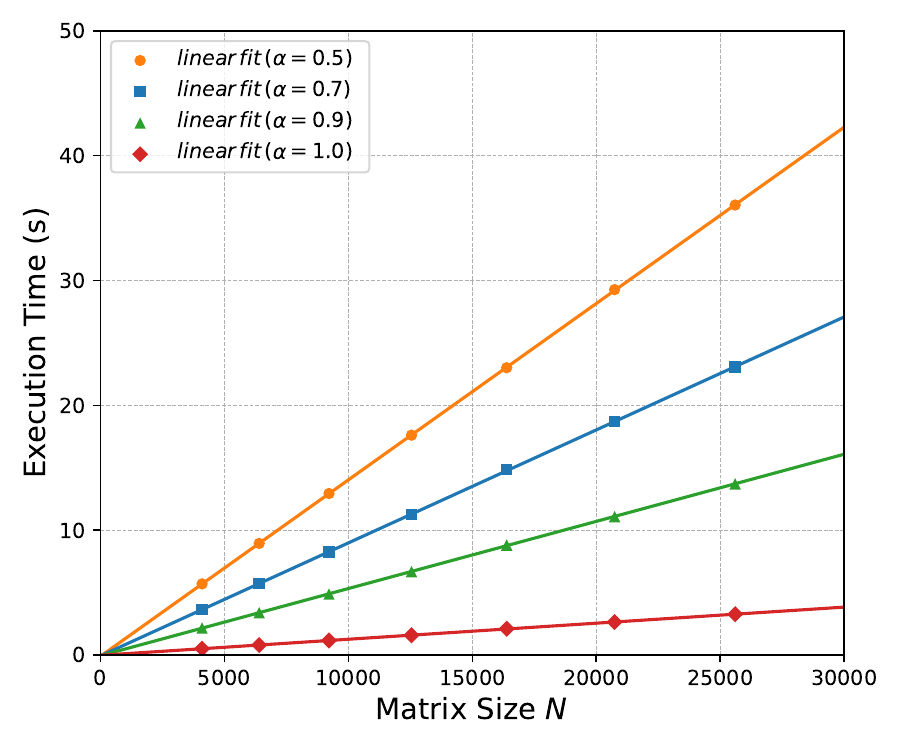}
	\caption{Elapsed time for solving the 2D diffusion equation as a function of the matrix size $N = m^2$ and the parameter $\alpha$ for $t~=~0.1$. The number of random paths was kept fixed at $10^6$.}
	\label{fig:poisson_scale_samples}
\end{minipage}
\end{figure}

Asymptotically when $n\gg 1$, $h_n\sim 2/\pi (n+1)$. Concerning the number of transitions, this is a random variable governed by a probability distribution $P_n$ and depends on the number of events occurring during the time interval $[0,t]$. Moreover, since the sojourn time in every state is Mittag-Leffler distributed, then the probability can be estimated, and it turns out to be the so-called fractional Poisson distribution~\cite{laskin_fractional_2003}, where the mean number of events $\bar{n}$ is
\begin{equation}
\bar{n}=\frac{m^2 t^{\alpha}}{\Gamma{(\alpha+1)}},
\label{eq:mean_num_states}
\end{equation}
and the probability of no events occurring during this time interval is \linebreak
$P_0 = E_\alpha (m^2 \, t^\alpha)$. When $m\to \infty$, a simple piecewise constant function can be used as an approximation and yields
\begin{equation}
 P_n \approx
    \begin{cases}   P_0,   & n\le \frac{1}{P_0} \\
                    0,     & n> \frac{1}{P_0}.
    \end{cases}
\end{equation}
Then, it follows that the probability $\xi$ is given by
\begin{equation}
\xi=\sum_{n=0}^\infty h_{2n} P_{2n}\approx P_0+P_0\sum_{n=1}^{1/P_0} h_{2n}.
\end{equation}
Using the known difference equation $\sum_{k=0}^{n-1}(k+x)^{-1}=\psi(x+n)-\psi(x)$, where $\psi(x)$ is the digamma function, we obtain
\begin{equation}
\xi\approx P_0+\frac{P_0}{\pi}\left[\psi \left(\frac{3}{2}+\frac{1}{P_0} \right)-\psi \left(\frac{5}{2} \right)\right].
\end{equation}
When $m \to \infty$, $P_0\sim [m^2 \, t^{\alpha} \, \Gamma(1-\alpha)]^{-1}$, and then using the asymptotic expansion properties of the digamma function it holds that $\xi\sim [m^{2} \, t^{\alpha} \, \Gamma(1~-~\alpha)]^{-1}$. Thus, asymptotically an estimation of the variance $\sigma^2$ is given by
\begin{equation}
\sigma^2\sim (c\, m^2)^2\xi=c^2\,m^2 \frac{t^{-\alpha}}{\Gamma(1-\alpha)}=c^2\,N \frac{t^{-\alpha}}{\Gamma(1-\alpha)}.
\label{eq:variance_poisson}
\end{equation}

This relation can be seen in Fig. \ref{fig:poisson_scale_var}. For a fixed $\alpha$ and $t$, the variance $\sigma^2$ scales linearly with the matrix size $N = m^2$. Likewise, for a fixed $N$, $t$, the variance $\sigma ^2$ is proportional to $[t^{\alpha} \, \Gamma(1~-~\alpha)]^{-1}$. Note that the points in the graph do not perfectly match the curve from (\ref{eq:variance_poisson}) due to the low number of cells $m$ used in the simulation. Still, it provides the best approximation for the variance $\sigma^2$ under these conditions.

Concerning the computational cost of the Monte Carlo algorithm, this depends on the number of random paths and the time for generating each path. Even though this time is random, we can readily assume it to be proportional to the mean time of events in (\ref{eq:mean_num_states}). Therefore, the computational cost of the algorithm is of the order of $O(N_p \, N \, t^\alpha)$, which is in good agreement with the results shown in Fig.~\ref{fig:poisson_time_scale} and~\ref{fig:poisson_scale_samples}. For a fixed number of random paths~$N_p$ and matrix size $N$, the execution time is proportional to $t^{\alpha}$ as shown in Fig.~\ref{fig:poisson_time_scale}. Note that the scale is $log-log$ and the slope of the curves coincides with $\alpha$. Similarly, when we kept fixed both the time $t$ and number of random paths $N_p$, the execution time grows linearly with the matrix size $N$ as shown in Fig.~\ref{fig:poisson_scale_samples}. It is worth mentioning that exponential random numbers are much cheaper to produce than ML random numbers, resulting in significantly lower execution times when $\alpha = 1$.

\afterpage{
\begin{figure}[t]
	\centering	\includegraphics[width=\textwidth]{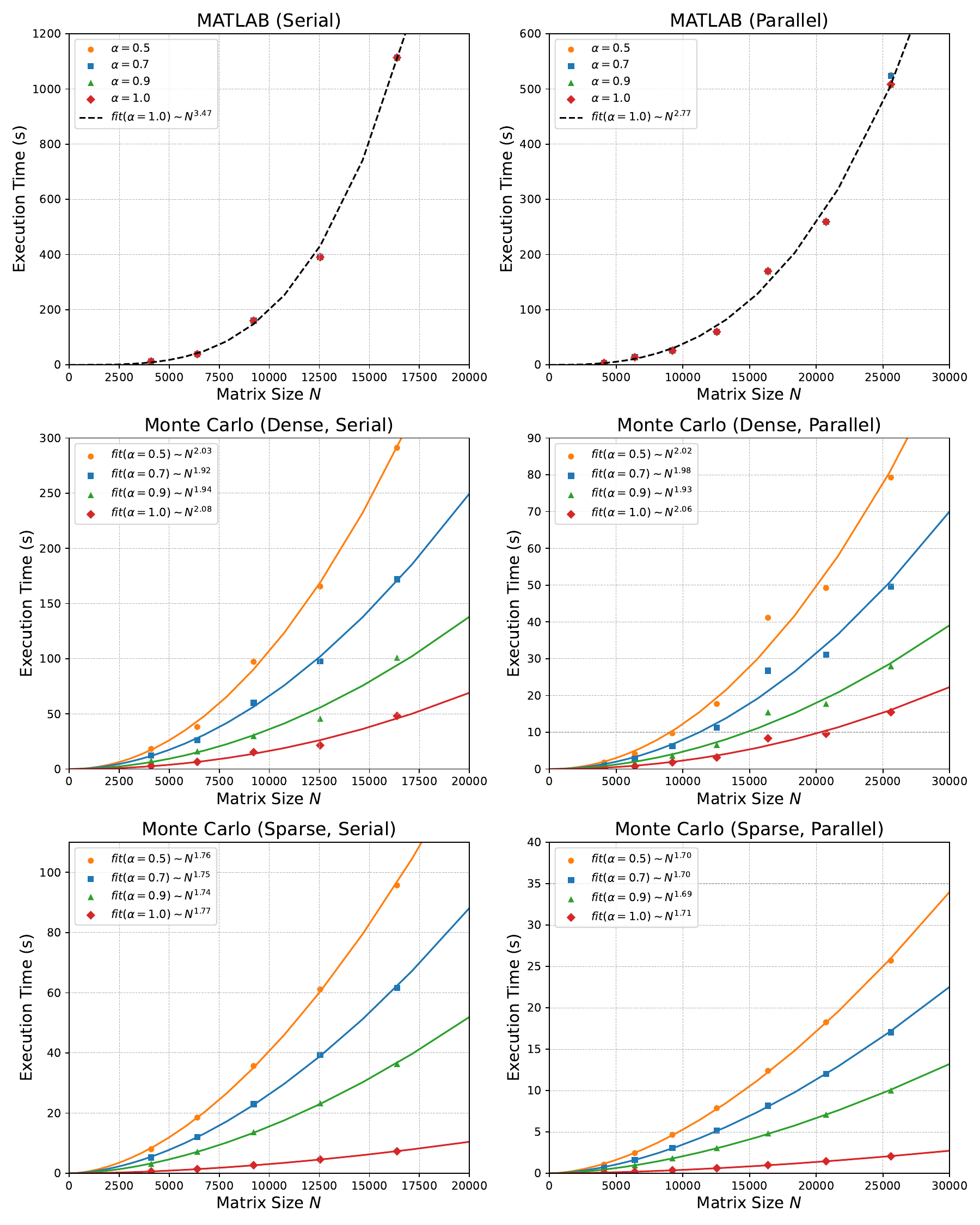}
	\caption{Serial (left) and parallel with 8 threads (right) execution times for solving the 2D diffusion equation as a function of the matrix size $N = m^2$ for $t = 0.1$. For the Monte Carlo algorithms, the accuracy $\epsilon$ was kept fixed at $3 \times 10^{-4}$.}
	\label{fig:poisson_scale}
\end{figure}
\clearpage
}

To compare the performance and scaling of our methods against a classical approach, we measure the serial and parallel execution times for different matrix sizes $N = m^2$ as shown in Fig.~\ref{fig:poisson_scale}. In terms of execution time, the Monte Carlo algorithm is several times faster than the {\tt matlab} code, especially when fully exploiting the sparse nature of the problem, while maintaining a reasonable precision (\ie, the maximum error in any simulation is below $3 \times 10^{-4}$).  When using 8 threads and considering a discrete mesh with $m = 128$, both \texttt{matlab} and \texttt{mc\_dense} shows a speedup of $6.5$ (or an efficiency of $81.25\%$), while \texttt{mc\_sparse} presents a speedup of $7.5$ ($95\%$ efficiency).

Since matrix $\mat{A}$ is real and symmetric, its Schur decomposition can be written as $\mat{A} = \mat{V}\mat{D}\mat{V}^\intercal$, where $\mat{D}$ is a diagonal matrix containing the eigenvalues of $\mat{A}$, and the columns of the matrix $\mat{V}$ the corresponding eigenvectors. In this case, the matrix function $f(\mat{A})$ can be simply evaluated as $f(\mat{A}) = \mat{V}f(\mat{D})\mat{V}^\intercal$~\cite{davies_schur-parlett_2003}.  MATLAB uses a QR algorithm to compute the Schur decomposition, directly calling the corresponding routine from the Intel Math Kernel Library (\eg, \texttt{dsteqr} for real symmetric matrices). The QR algorithm has a well-known computational complexity of order $O(N^3)$.

In comparison, the Monte Carlo algorithm requires an order of $O(N^2)$ operations to solve the diffusion equation, as shown in Fig.~\ref{fig:poisson_scale}. Considering that the variance $\sigma^2$ and, consequently, the statistical error grow with the matrix size $N$, the number of random paths must increase proportionally to the matrix size in order to keep the same level of accuracy for the numerical solution.

\subsection{Time-fractional convection-diffusion equation} \label{sec:fem}

The second example we discuss in this paper consists of solving a time-fractional convection-diffusion equation:
\begin{equation}
\label{eq:convec_diff}
D^\alpha_t u(\vec{x}, t) = c\nabla^2 u(\vec{x}, t) + \nu\nabla u(\vec{x}, t),
\end{equation}
with the respective boundary and initial conditions
\begin{equation}
\label{eq:diff_initial}
u(\vec{x}, t)\lvert_{\partial\Omega}\ = g(\vec{x}, t) \qquad u(\vec{x}, 0) = f(\vec{x}),
\end{equation}
where $c$ is the diffusion coefficient and $\nu$ the velocity field. After applying the standard Galerkin finite element method~\cite{zienkiewicz_finite_2013} to the discretized nodes $x_i$ for $i = 0, 1, \dots, n$, we obtain the following system of equations:
\begin{equation}
\label{eq:fem}
\mat{B} D^\alpha_t \vec{u} = \mat{K} \vec{u} + \vec{h}, \qquad \vec{u}(0) = \vec{u}_0,
\end{equation}
where $\vec{u} = \{u(x_1, t), u(x_2, t), \dots, u(x_n, t)\}$, $\mat{B}$ is the assembled mass matrix, $\mat{K}$ is the corresponding assembled stiffness matrix and $\vec{h}$ is the load vector. The diffusion coefficient $c$, the velocity field $\nu$, and the boundary conditions are already included in these matrices and vectors. To simplify the computation, the mass matrix was lumped \cite{zienkiewicz_finite_2013}, transforming the mass matrix into a diagonal matrix.

The solution for the inhomogeneous systems of fPDEs in (\ref{eq:fem}) can be written in terms of the Mittag-Leffler function as follows

\begin{equation}
\label{eq:fem_solution}
\vec{u}(\vec{x}, t) = E_\alpha(\mat{A} \, t^\alpha) \ \vec{u}_0(\vec{x}) + \int_0^t ds (t - s)^{\alpha - 1} E_\alpha(\mat{A} \, (t - s)^\alpha) \vec{v} ,
\end{equation}
with $\mat{A} = \mat{B}^{-1}\mat{K}$ and $\vec{v} = \mat{B}^{-1}\vec{h}$. Note that the second term on the right side depends on the boundary data, and in particular for numerical purposes, it will require approximating a definite integral. This will entail another source of error along with the aforementioned statistical error. Since the goal of this manuscript is to focus solely on computing numerically the ML function and analyzing the associated statistical error, here we set a null value for the Dirichlet boundary conditions $f(\vec{x}) = 0$, and hence $\vec{h} = 0$. Clearly, the general case deserves further investigation, a detailed analysis is left for a future manuscript.

\begin{figure}
\centering
\begin{minipage}[c]{0.475 \textwidth}
	\centering \includegraphics[width=\textwidth]{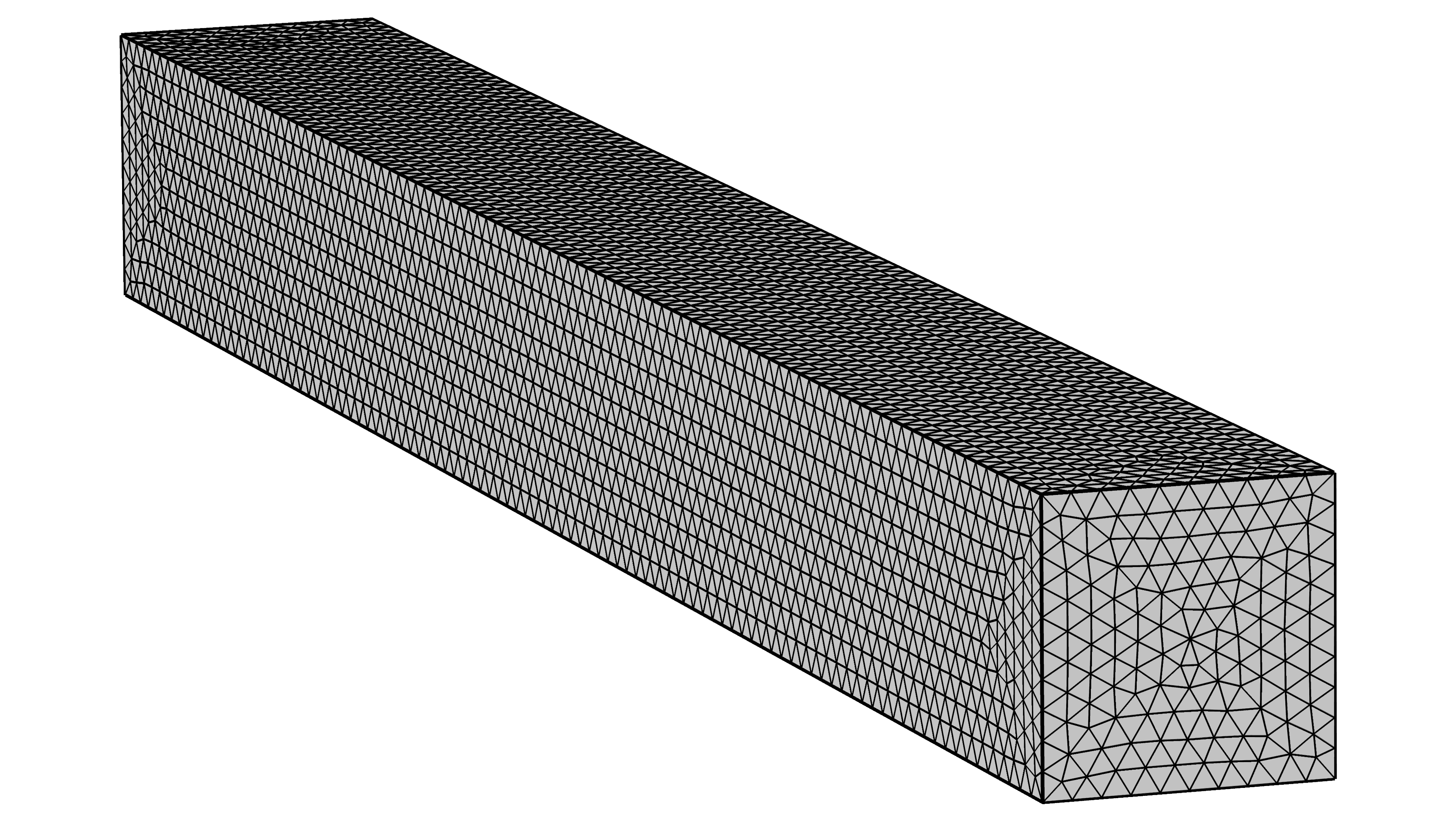}
	\caption{Computational mesh of the domain where the fractional convection-diffusion equation is solved.}
	\label{fig:fem_mesh}
\end{minipage}%
\quad
\begin{minipage}[c]{0.475 \textwidth}
	\centering	\includegraphics[width=\textwidth]{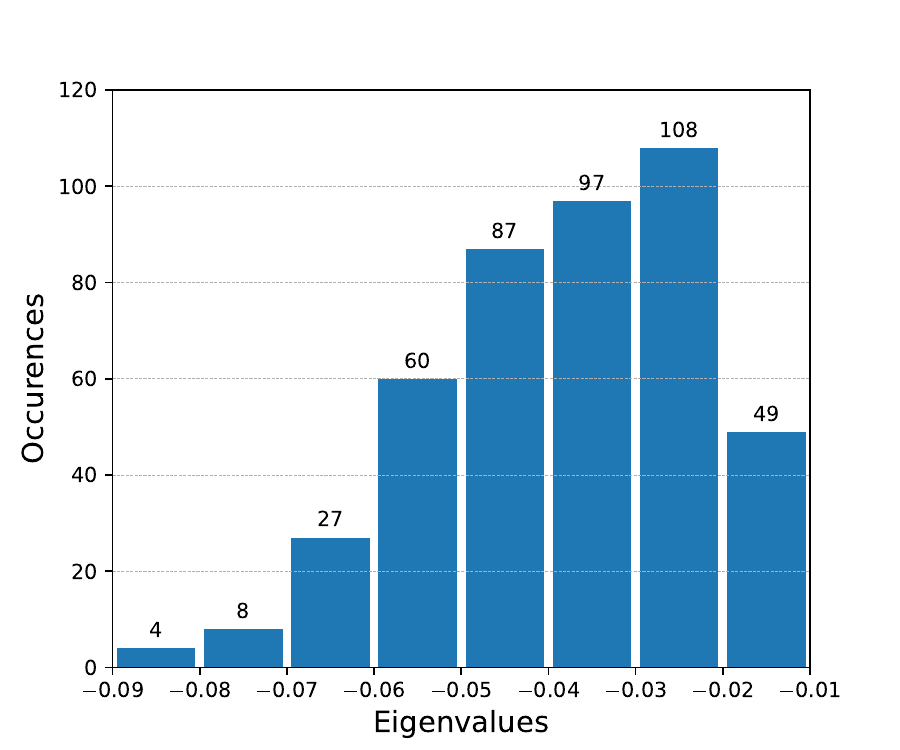}
	\caption{Histogram of the eigenvalues of $\mat{A} = \mat{B}^{-1}\mat{K}$ from (\ref{eq:fem_solution}). Here  the scale parameter was set to $s = 30$, which results in a matrix of size $N~=~453$.}
	\label{fig:fem_eigen_hist}
\end{minipage}
\end{figure}

For the domain, we consider here a block of aluminium of $2s \times 2s \times 20s$ ($W \times H \times L$, in cm) with a scale parameter $s$, an isotropic diffusion coefficient $c~=~9.4 \times 10^{-5} m^2 / s$~\cite{parker_flash_1961}, a field velocity $\nu = 0$ and a Dirichlet boundary condition $u = 0$ on the surface of the block. The finite-element mesh in Fig.~\ref{fig:fem_mesh} was generated using the scientific application tool \textit{COMSOL}~\cite{comsol} considering a maximum and minimum element size equal to $20$ cm and $0.1$ cm, respectively. Based on this mesh, we generated the corresponding finite-element method (FEM) matrices and vectors. For the initial conditions, we use again a discrete impulse located at the position $(s, s, 10s)$. It is worth mentioning that the resulting matrix is asymmetric, having entries entirely arbitrary in both sign and magnitude and in this sense, this can be considered as a much more complex example than the case analysed in the previous section.

\begin{table}[t]
\centering
\caption{The execution time of the \texttt{matlab} code for computing the solution of the 3D convection-diffusion equation for different values of $t$ and $\alpha$ considering a scale parameter $s = 30$ (the total number of nodes in the computational mesh was $453$). The parameters of \texttt{funm} were kept at their default values. The number in parenthesis indicates the order of the derivative of the largest block.}
    \begin{tabular}[c]{ccccc}\toprule
    & $\alpha = 1$ & $\alpha = 0.9$ & $\alpha = 0.7$ & $\alpha = 0.5$\\\midrule
    $t = 20$ & $6.43\ (19)$ & FAIL & FAIL & $919\ (21)$\\\midrule
    $t = 40$ & $5.81\ (22)$ & FAIL & FAIL & $965\ (24)$\\\midrule
    $t = 60$ & $5.55\ (24)$ & FAIL & FAIL & $967\ (24)$\\\midrule
    $t = 80$ & $6.15\ (27)$ & FAIL & FAIL & $971\ (25)$\\\midrule
    $t = 100$ & $6.78\ (30)$ & FAIL & FAIL & $982\ (27)$\\\bottomrule
    \end{tabular}
\label{table:fem_matlab}
\end{table}

Similarly to the previous example, we examine the eigenvalues of matrix $\mat{A}$ to gain some insight about its numerical properties. In this case, the eigenvalues were calculated through the command \texttt{eigs} of MATLAB, which corresponds in essence to the Arnoldi method. Fig.~\ref{fig:fem_eigen_hist} depicts the histogram for the eigenvalues of the matrix $\mat{A}$. The eigenvalues are not only small but are also highly clustered between $-0.02$ and $-0.05$, which is a problem for the \texttt{matlab} code of \cite{garrappa_computing_2018}. On one hand, this leads to very large blocks during the partitioning of the Schur form of $\mat{A}$ in the Schur-Parlett algorithm~\cite{davies_schur-parlett_2003, garrappa_computing_2018}. The evaluation of the Taylor series is particularly expensive for large blocks since it requires order $O(m^4)$ operations, where $m$ is the size of the block. On the other hand, the numerical evaluation of high-order derivatives of the Mittag-Leffler function often entails very large numbers even for small input values due to the rapidly increasing factorial $(x)_k = x(x - 1) \dots (x - k + 1)$:
\begin{equation}
\label{eq:ml_derivative}
   \frac{d^k}{dz^k} E_{\alpha}(z) = \sum_{j = k}^{\infty}{\frac{(j)_k}{\Gamma(\alpha j + 1)}z^{j - k}}, \quad k \in \mathbb{N}.
\end{equation}

Indeed, MATLAB reports an infinite derivative when trying to compute the ML function for $\alpha = 0.7$ and the algorithm fails to converge to a solution for $\alpha = 0.9$ even after computing $250$ terms of the Taylor series. This is shown in Table~\ref{table:fem_matlab}. The \texttt{matlab} algorithm was able to reach a solution for $\alpha = 0.5$ and $\alpha = 1$.

In the Schur-Parlett algorithm, the Schur form of $\mat{A}$ are divided into blocks according to its eigenvalues. In particular, if the absolute difference between two eigenvalues is less than a tolerance $\delta$, they will be assigned to the same block. With the default tolerance $\delta = 0.1$, almost all eigenvalues of $\mat{A}$ will be grouped into a single, large block. Even after decreasing the tolerance to $\delta = 0.01$, most eigenvalues are still concentrated into a single block. Note that the value of $\delta$ cannot be too small, otherwise, the two distinct blocks may contain close eigenvalues, causing the Parlett's recurrence to break down due to numerical cancellations~\cite{davies_schur-parlett_2003, higham_accuracy_2002}. It is worth mentioning that the derivatives of the ML function are calculated serially and, thus, the \texttt{matlab} code shows very little speedup when using multiple threads.

\afterpage{
    \begin{figure}[H]
        \centering	\includegraphics[width=\textwidth]{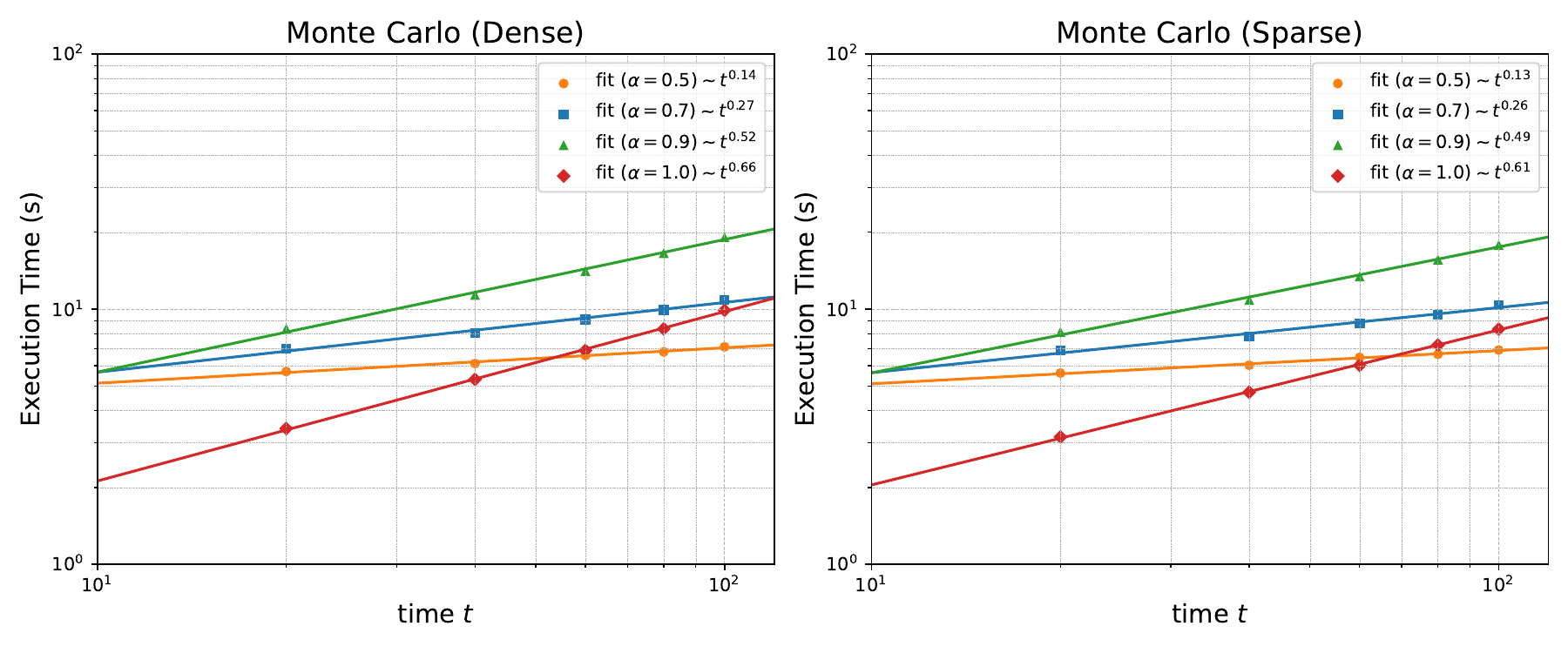}
        \caption{Serial execution time of the probabilistic algorithm for computing the solution of the 3D convection-diffusion equation for different values of $t$ and $\alpha$, considering a scale parameter $s = 30$ (the total number of nodes in the computational mesh was $453$). The number of random paths was kept fixed at $10^8$, resulting in a precision $\epsilon \approx  10^{-4}$.}
        \label{fig:fem_time_scale}
    \end{figure}

        \begin{figure}[H]
    \centering
    \begin{minipage}[t]{0.475 \textwidth}
    	\centering	\includegraphics[width=\textwidth]{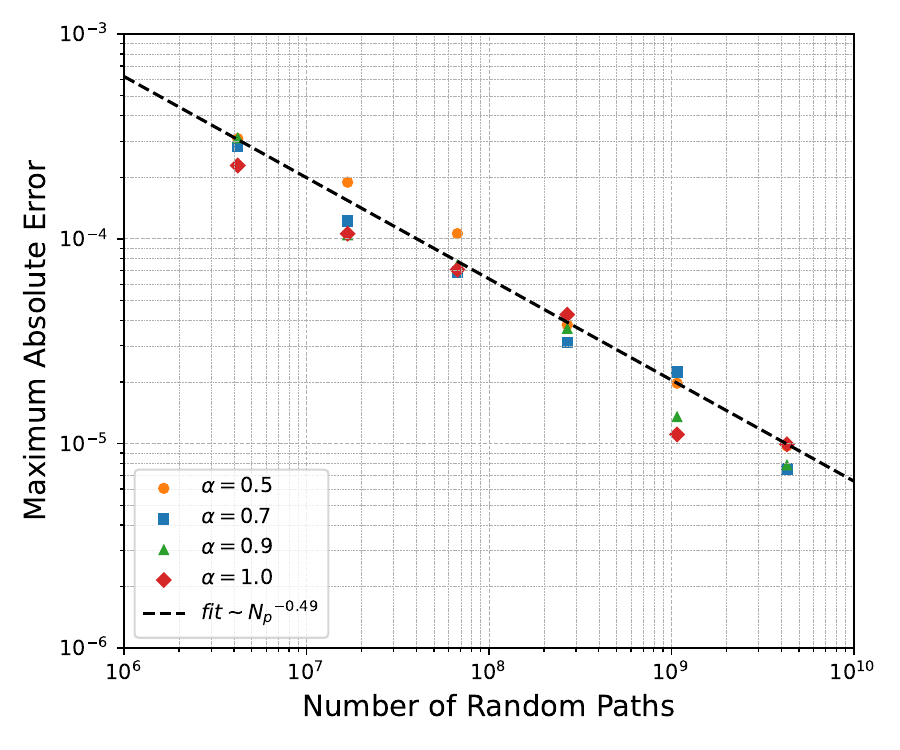}
    	\caption{Maximum absolute error of the \texttt{mc\_sparse} algorithm as a function of the number of random paths when solving the 3D convection-diffusion equation for time $t = 40$ and scale parameter $s = 100$ (the total number of nodes in the computational mesh was $25 120$). }
    	\label{fig:fem_err}
    \end{minipage}%
    \quad
    \begin{minipage}[t]{0.475 \textwidth}
    	\centering	\includegraphics[width=\textwidth]{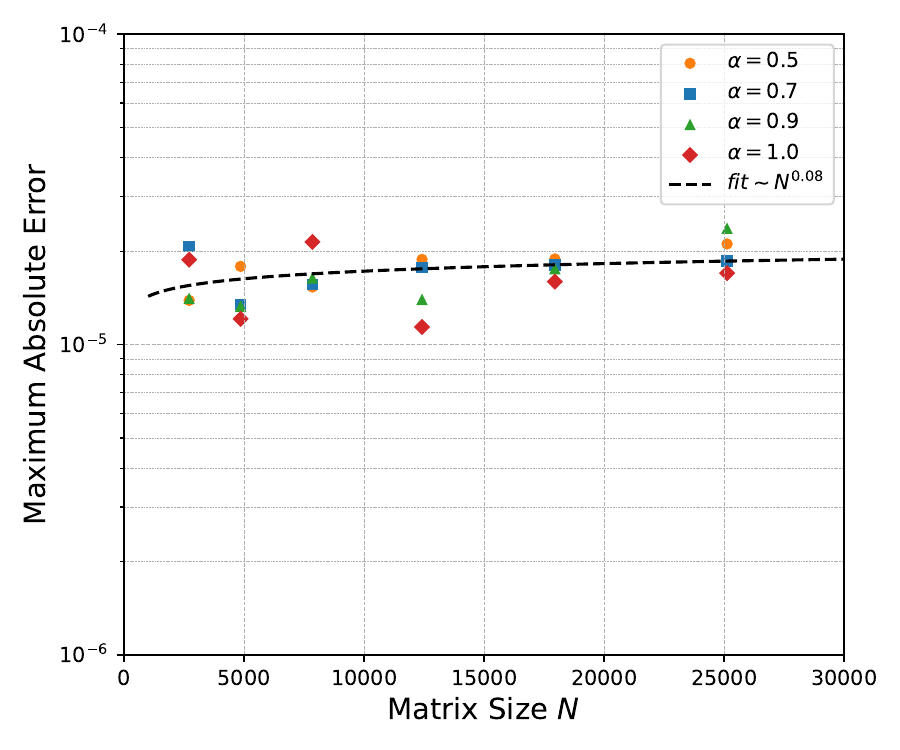}
    	\caption{Maximum absolute error of the \texttt{mc\_sparse} algorithm as a function of the number of nodes $N$ in the computational mesh when solving the 3D convection-diffusion equation for $t = 40$. The number of random paths was kept fixed at $10^9$.}
    	\label{fig:fem_scale_err}
    \end{minipage}
    \end{figure}

    \newpage
}

Concerning the computational cost of the Monte Carlo method for solving this problem, recall that this depends on the number of random paths and the mean number of events that occurred during the prescribed time interval. The former is chosen according to the desired level of accuracy for the numerical solution. As was already explained in the previous example, the numerical error depends not only on the number of random paths but also on the variance, which depends on the specific entries of the matrix $\mat{A}$ and the input vector $\vec{u}$. The more heterogeneous they are, the larger the variance will be, and consequently more random paths will be required to attain the prescribed accuracy. Regarding the mean number of events, this depends on the specific values of the diagonal entries of $\mat{A}$, being larger for larger entries.

In contrast to the \texttt{matlab} code which depends strongly on the distribution of eigenvalues, the Monte Carlo method does not show explicit dependency on them. For this reason, the stochastic method was able to compute the solution regardless of the values of $t$ or $\alpha$. The execution time of {\tt mc\_dense} and {\tt mc\_sparse} for different values of $t$ and $\alpha$ are shown in Fig. \ref{fig:fem_time_scale}. Regarding the parallel performance, {\tt mc\_dense} and {\tt mc\_sparse} usually achieve a speedup between $5$ and $6$ when using $8$ threads.

\begin{figure}[t]
    \centering	\includegraphics[width=\textwidth]{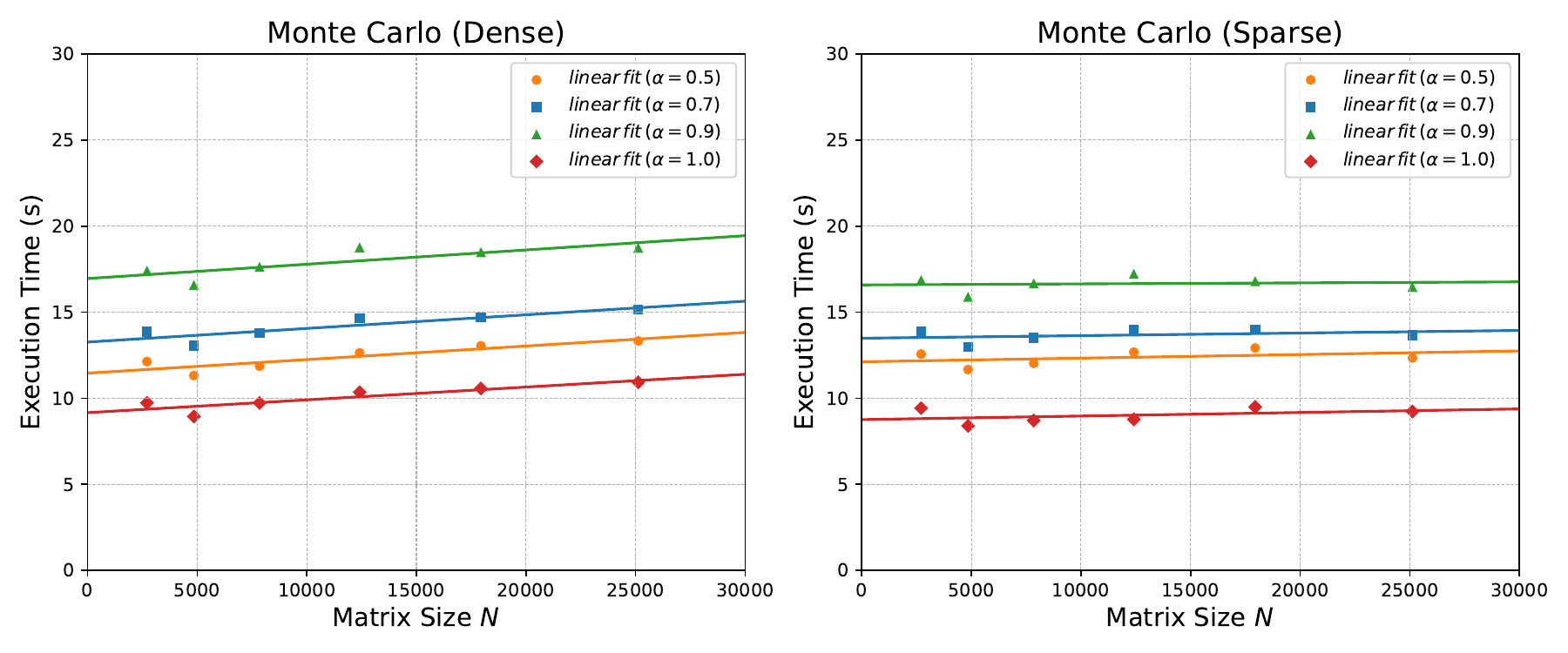}
    \caption{Parallel execution time of the Monte Carlo method as a function of the number of nodes $N$ in the computational mesh when computing the solution of the 3D convection-diffusion equation for $t = 40$. The number of random paths was kept fixed at $10^9$, resulting in an error $\epsilon \approx 2 \times 10^{-5}$.}
    \label{fig:fem_scale}
\end{figure}

Figures~\ref{fig:fem_err} and~\ref{fig:fem_scale_err} show the maximum absolute error as a function of the number of random paths and the matrix size $N$, respectively. Note that the numerical error of the method $\epsilon$ decreases with the square root of the number of random paths as theoretically expected.

In the previous example, the domain was fixed to $\Omega~=~[-1, 1]^2$, and the matrix size was enlarged increasing the number of grid points inside the domain. As a result, the magnitude of the diagonal entries of the matrix increases accordingly, and so do the mean number of events and the variance.  We now consider a case where the distance between two nodes in the finite-element mesh is kept constant when the domain is conveniently scaled in order to increase the matrix size. This leads to a matrix with diagonal entries independent of the matrix size, and therefore, the execution time and error of the Monte Carlo algorithm only depend on the number of random paths. This is shown in Fig.~\ref{fig:fem_scale_err} and~\ref{fig:fem_scale}. Note however that the algorithm needs to perform a binary search for selecting the next state of the random path, which causes the execution time to slightly grow as the matrix size increases. This has a negligible effect on the sparse implementation due to a low number of nonzero entries per matrix row.

\subsection{Complex Geometry}

\begin{figure}[t]
\centering
\begin{minipage}[t]{0.475 \textwidth}
    \centering	\includegraphics[width=\textwidth]{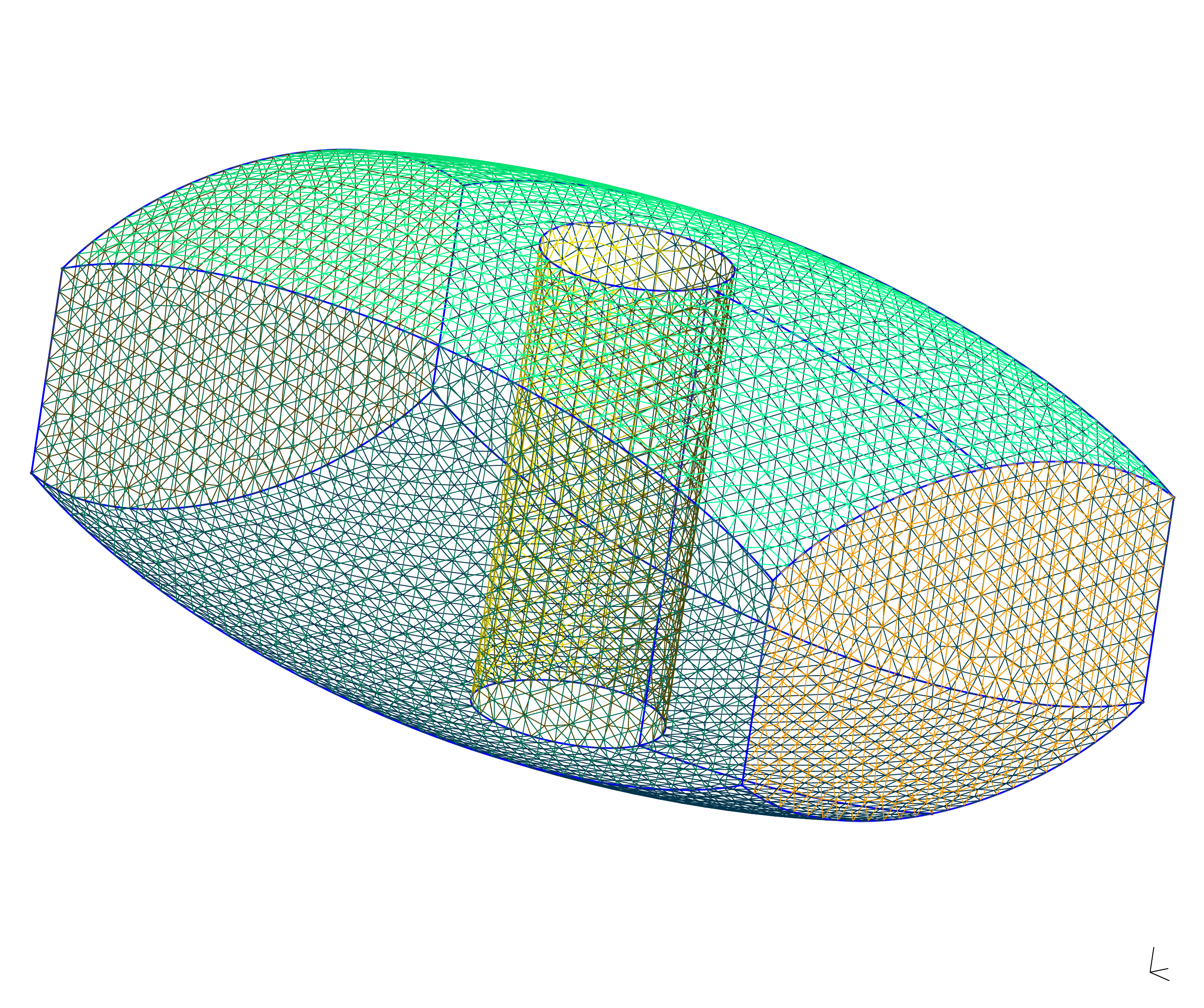}
    \caption{Computational mesh of the complex geometry for solving the fractional convection-diffusion equation.}
    \label{fig:spaceship}
\end{minipage} %
\quad
\begin{minipage}[t]{0.475 \textwidth}
    	\centering	\includegraphics[width=\textwidth]{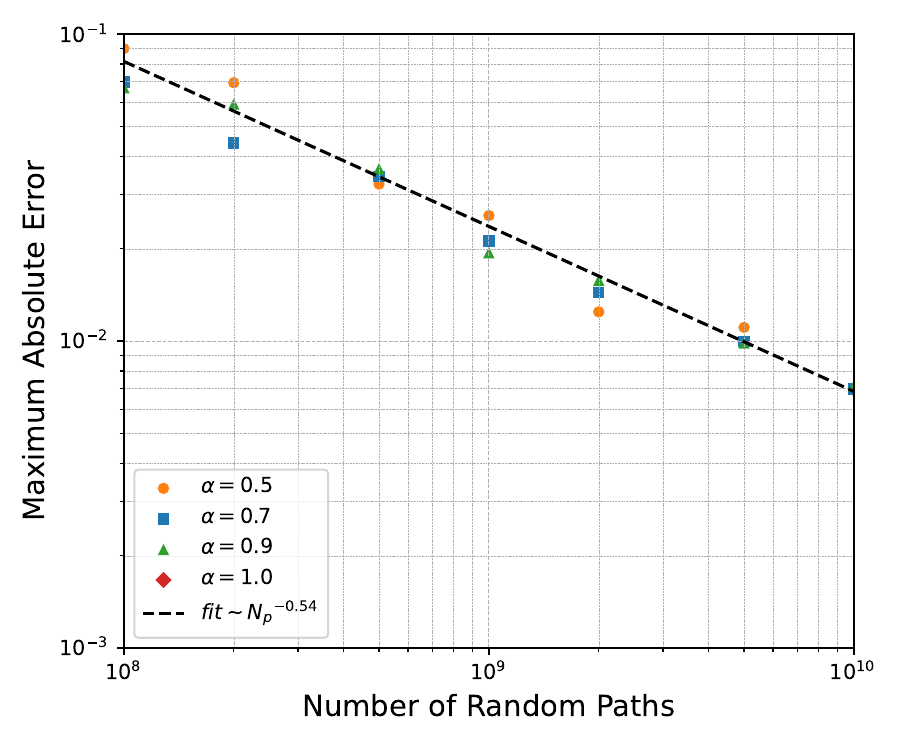}
    	\caption{Maximum absolute error of the \texttt{mc\_sparse} algorithm as a function of the number of random paths when solving the 3D convection-diffusion equation for a complex geometry for $t = 4$. }
    	\label{fig:spaceship_err}
\end{minipage}
\end{figure}

Next, we solved the time-fractional convection-diffusion equation (\ref{eq:convec_diff}) over a more complex
geometry (Fig. \ref{fig:spaceship}). The dimensions of the object are $1 \times 0.6 \times 0.6$ ($W \times H \times L$, in m) and the size of the discrete elements is set to $0.025$, such that the resulting mesh has $14 216$ nodes and $197 344$  elements. Here, we also consider the same isotropic diffusion coefficient $c~=~9.4 \times 10^{-5} m^2 / s$~\cite{parker_flash_1961} and field velocity $\nu = 0$ as before, but with the Neumann boundary conditions $\nabla \vec{u} \cdot \vec{n} = 0$ instead of Dirichlet as before. As the initial condition, we set all nodes within a region in the domain to $1$. More specifically, the nodes within the following range: $x~=~[-0.3, -0.25), y~=~[-0.1, 0.1), z~=~[-0.1, 0.1)$ with the origin set at the centre of the object.

Since the resulting matrix is quite large and the \texttt{matlab} code struggled to solve the equation on a simple mesh, we could only run the sparse version of our algorithm (\texttt{mc\_sparse}) in this example. Fig. \ref{fig:spaceship_err} shows how the maximum absolute error scales with the number of random paths. Although the numerical error in this example is larger than others, it still decreases with $\sqrt{N_p}$, such that the numerical solution of our algorithm will converge to the true solution with a sufficient large number of samples.

\subsection{Distributed-Memory Performance} \label{sec:perf_dist}

\begin{figure}[t]
\centering
\begin{minipage}[t]{0.475 \textwidth}
    \centering	\includegraphics[width=\textwidth]{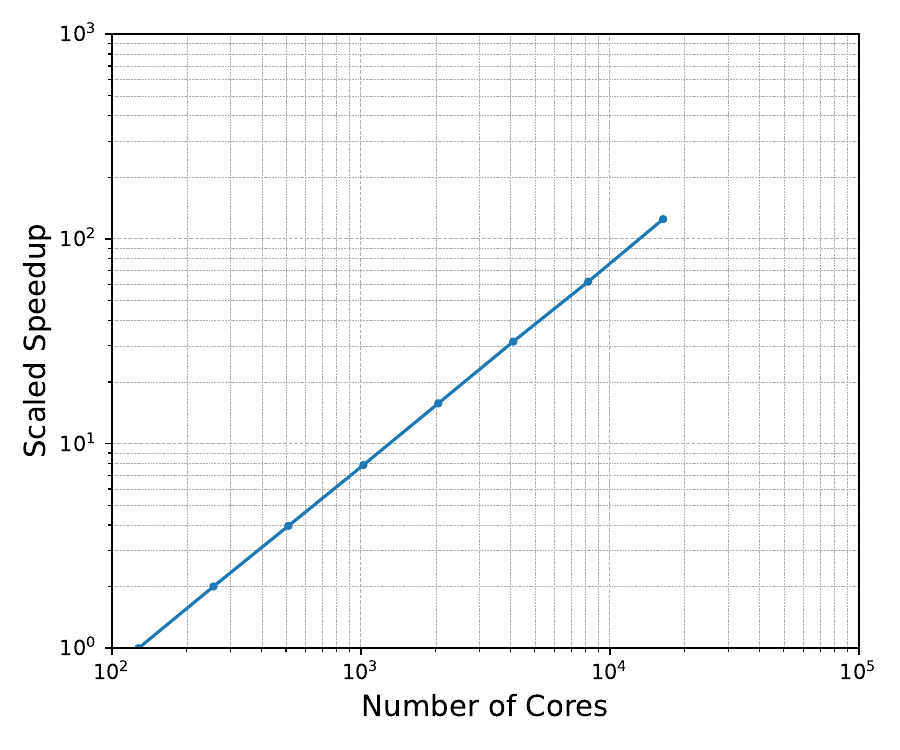}
    \caption{Weak scaling of \texttt{mc\_sparse} in the Karolina supercomputer. This test solves a 2D diffusion equation for $t = 0.1$, $\mu = 1$ and $m = 1024$ ($N = m^2 = 1,048,576$). The number of random paths was chosen based on the number of cores, beginning with $2 \times 10^6$.}
    \label{fig:poisson_mpi}
\end{minipage} %
\quad
\begin{minipage}[t]{0.475 \textwidth}
    \centering	\includegraphics[width=\textwidth]{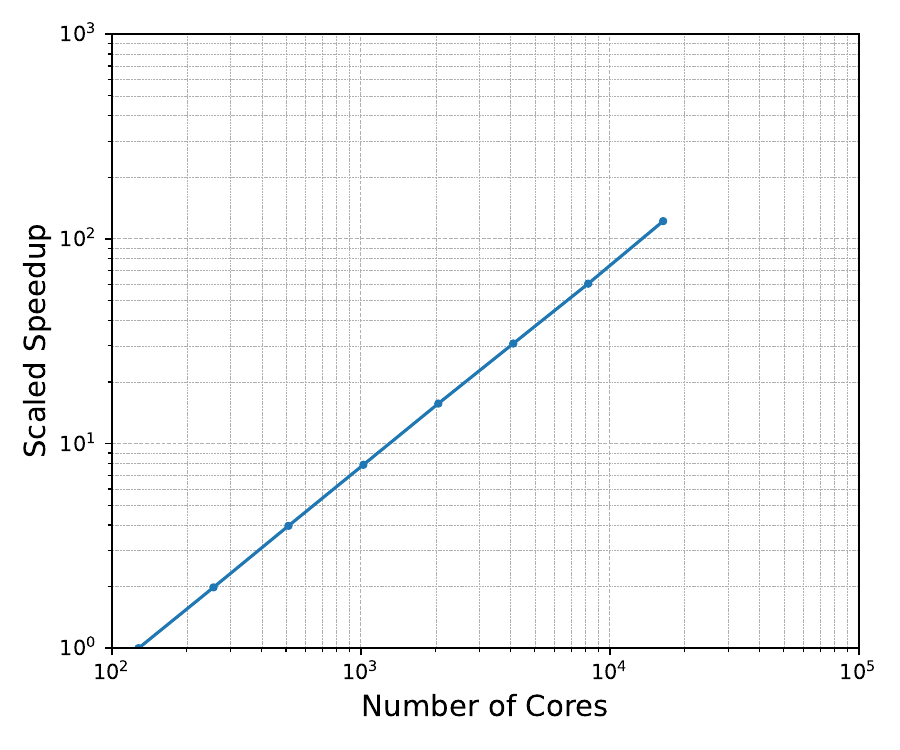}
    \caption{Weak scaling of \texttt{mc\_sparse} in the Karolina supercomputer. This test solves a 3D convection-diffusion equation for $t = 40$ considering a scalar parameter $s = 300$ (the total number of nodes in the computational mesh was $753,675$). The number of random paths was chosen based on the number of cores, beginning with $10^{11}$.}
    \label{fig:fem_mpi}
\end{minipage}
\end{figure}

The distributed-memory tests were carried out at the Karolina Supercomputer located in IT4Innovations National Supercomputing Centre. Each computational node in the cluster contains two AMD 7H12 64C~@2.6GHz CPUs and 256GB of RAM, running CentOS 7. There are two processes per node (one for each processor), each one with 64 threads. Since there is no freely available code capable of computing the action of a Mittag-Leffler function over a vector that is suitable for distributed-memory systems, this section only contains the results for the Monte Carlo method (\texttt{mc\_sparse}). The C++ code was compiled with GCC 12.1.0 and OpenMPI 4.1.4 with the {\tt -O3} and {\tt -march=native} flags.

Figures~\ref{fig:poisson_mpi} and~\ref{fig:fem_mpi} show the weak scaling of \texttt{mc\_sparse} between 1 and 128 nodes ($16,384$ cores), considering a single computational node as the baseline. While the scaling of the domain size is typically used to increase the workload, for the specific FEM problem in (\ref{eq:fem}), as it was shown in the previous section, this does not affect the computational time, and therefore does not provide any useful insight about the scalability of the algorithm. Instead, we increased the workload by scaling the number of random paths according to the number of nodes used. In doing so the statistical error reduces, improving the accuracy of the computed solution.

Since all random paths in the Monte Carlo method are independent, the task of generating them can be done in parallel and the parallel code only needs to communicate at the end of the simulation to combine the results. This is the reason that the \texttt{mc\_sparse} shows perfect scalability in both examples.

\section{Conclusion} \label{sec:conclusion}

In this paper, we propose a novel stochastic method for solving time fractional partial differential equations. These equations are already being used for modelling natural phenomena subject to memory effects, and microscopically are typically described by non-Markovian processes. Fractional equations are capable of capturing such effects due to the inherent non-locality of the operator. As a consequence, the classical numerical schemes often based on time-stepping suffer from heavy memory storage requirements since the numerical solution of the fPDEs at a current time depends on all preceding time instances. This can be even worse when dealing with high-dimensional problems, degrading significantly the performance of the corresponding numerical algorithms.

The main advantage of the proposed method rests on the fact that it allows for the computation of the solution at a single point of a given domain, or equivalently a single entry of the corresponding vector solution. Moreover, the numerical algorithm is not based on any time-stepping scheme, and therefore the solution is obtained without the need to store previous results. Rather, the solution is computed through an expected value of a functional of random processes which resembles the non-Markovian process found in the microscopic description of the phenomenon, and hence exploits somehow naturally the non-locality of the fractional operators.

Furthermore, since the method is based on Monte Carlo it inherits all the known advantages from a computational point of view, such as the comparative ease of implementation in parallel, fault-tolerance, and in general being well suited for heterogeneous architectures. In fact, our parallel implementation was able to solve large-scale problems efficiently in both shared-memory and distributed-memory systems, demonstrating the versatility and scalability of the probabilistic method compared to the classical numerical schemes.

\section*{Acknowledgments}

\begin{sloppy}
This work was supported by national funds through FCT, Fundação para a Ciência e a Tecnologia, under projects URA-HPC PTDC/08838/2022 and UIDB/50021/2020 (DOI:10.54499/UIDB/50021/2020) and grant 2022.11506.BD..
JA was funded by  Ministerio de Universidades and specifically  the requalification program of the Spanish University System 2021-2023 at the Carlos III University.
\end{sloppy}



\bibliographystyle{elsarticle-num}
\bibliography{bibliography}

\end{document}